\documentclass[12pt,a4paper]{amsart}
\usepackage{amsfonts, amssymb, amsmath, amsthm}
\usepackage{graphicx,latexsym,eufrak,mathrsfs}
\usepackage[hmargin=1.2in, vmargin=1.2in]{geometry}

\newtheorem{theorem}{Theorem}[section]
\newtheorem{proposition}[theorem]{Proposition}
\newtheorem{lemma}[theorem]{Lemma}
\newtheorem{corollary}[theorem]{Corollary}

\theoremstyle{definition}

\theoremstyle{remark}
\newtheorem{remark}[theorem]{Remark}

\numberwithin{equation}{section}

\begin{document}

\title[Discrete characterizations of wave front sets] {Discrete characterizations of wave front sets of Fourier-Lebesgue and quasianalytic type}

\author[A. Debrouwere]{Andreas Debrouwere}
\thanks{A. Debrouwere gratefully acknowledges support by Ghent University, through a BOF Ph.D.-grant.}

\address{Department of Mathematics, Ghent University, Krijgslaan 281 Gebouw S22, 9000 Gent, Belgium}
\email{Andreas.Debrouwere@UGent.be}

\author[J. Vindas]{Jasson Vindas}
\thanks{The work of J. Vindas was supported by Ghent University through the BOF-grant 01N01014.}
\address{Department of Mathematics, Ghent University, Krijgslaan 281 Gebouw S22, 9000 Gent, Belgium}
\email{jvindas@cage.UGent.be}

\subjclass[2010]{Primary 35A18, 42B05. Secondary 46F05.}
\keywords{Wave front sets; Fourier series; ultradifferentiable functions; quasi\-analytic classes; Fourier-Lebesgue spaces; ultradistributions.}

\begin{abstract}
We obtain discrete characterizations of wave front sets of Fourier-Lebesgue and quasianalytic type. It is shown that the microlocal properties of an ultradistribution can be obtained by sampling the Fourier transforms of its localizations over a lattice in $\mathbb{R}^{d}$. In particular, we prove the following discrete characterization of the analytic wave front set of a distribution $f\in\mathcal{D}'(\Omega)$. Let $\Lambda$ be a lattice in $\mathbb{R}^{d}$ and let $U$ be an open convex neighborhood of the origin such that $U\cap\Lambda^{*}=\{0\}$. The analytic wave front set $WF_{A}(f)$ coincides with the complement in $\Omega\times(\mathbb{R}^{d}\setminus\{0\})$ of the set of points $(x_0,\xi_0)$ for which there are an open neighborhood $V\subset \Omega\cap (x_0+U)$ of $x_0$,  an open conic neighborhood $\Gamma$ of $\xi_0$, and a bounded sequence $(f_p)_{p \in \mathbb{N}}$  in $\mathcal{E}'(\Omega\cap (x_0+U))$ with $f_p= f$ on $V$ such that for some $h > 0$
\[
  \sup_{\mu \in \Gamma \cap \Lambda} |\widehat{f_p} (\mu)| |\mu|^p  \leq h^{p+1}p!\:, \qquad \forall p \in \mathbb{N}. 
   \]

\end{abstract}

\maketitle


\section{Introduction}
In this article we provide discrete characterizations of wave front sets of various types. We shall show that the microlocal properties of (ultra)distributions are completely determined  by the decay properties of the restrictions of Fourier transforms of their localizations to an arbitrary lattice in $\mathbb{R}^{d}$. To this end, we also study Fourier series expansions of ultradistributions. The paper refines and extends earlier results on \emph{toroidal} wave front sets from \cite{Maksi,RT}. 

Wave front sets play a fundamental role in the analysis of propagation of singularities of solutions to partial differential equations. The classical wave front set, originally introduced by H\"{o}rmander \cite{Hormander1971}, is defined with respect to $C^{\infty}$-smoothness, but this concept can be refined to include wave front sets with respect to other smoothness scales, such as Denjoy-Carleman classes and, in particular, analyticity  \cite{Hormander, Komatsutwee, Rodino}. More recently, wave front sets with respect to Fourier-Lebesgue spaces and other classes of Banach and Fr\'{e}chet spaces have been introduced and studied in \cite{c-j-t1,c-j-t2,Pilipovicthree,Pilipovictwo}. All of these variants have been systematically applied to the study of regularity properties of various classes of pseudo-differential operators and semilinear equations. The notion of wave front set has also relevant applications in mathematical physics, see the expository article \cite{b-d-h} for an overview.

The question of whether the wave front set of a distribution can be described in a discrete fashion goes back to Ruzhansky and Turunen \cite{RT}. Naturally, this is a very important question from a computational point of view. 
Their work is motivated by the study of (global) quantization of periodic pseudo-differential operators through Fourier series \cite{RT,Ruzh}. Denoting as $WF^{\mathbb{T}^{d}}(f)$ the ($C^{\infty}$) toroidal wave front set of a distribution $f$ on the $d$-dimensional torus $\mathbb{T}^{d}$, they have established the equality \cite[Sect. 7] {RT}
\begin{equation}
\label{RTwavefront}
WF^{\mathbb{T}^{d}}(f)=WF(f)\cap(\mathbb{T}^{d}\times \mathbb{Z}^{d}),
\end{equation}
where $WF(f)$ stands for the classical H\"{o}rmander wave front set when regarding $f$ as a distribution on $\mathbb{R}^{d}$. The latter equality further extends to Sobolev-type and Gevrey wave front sets, as recently shown in \cite{d-m-s,Maksi}. It should also be mentioned that Rodino and Wahlberg \cite{Rodino-W} and Johansson et al. \cite{Johansson} have investigated discrete definitions of microregularity properties of distributions via Gabor frames. 

Our goal here is to generalize (\ref{RTwavefront}) in several directions. On the one hand, we prove that the equality (\ref{RTwavefront}) remains valid for wave front sets of Fourier-Lebesgue and quasianalytic types. On the other hand, we replace $\mathbb{Z}^{d}$ in (\ref{RTwavefront}) by an arbitrary lattice in $\mathbb{R}^{d}$. Since we will actually reformulate the equality (\ref{RTwavefront}) in slightly different terms, namely, in terms of discretized estimates for Fourier transforms, the arbitrariness of the lattice strengthens the potential computational content of the results.

We now briefly describe the content of the paper and state some samples of our results. We mention that we will work with both the Beurling-Bj\"{o}rck \cite{Bjorck} and the Komatsu \cite{Komatsu} approach to the theory of ultradifferentiable functions and ultradistributions (see the preliminary Section \ref{preliminaries} for the notation). It turns out that many of the arguments employed in the article depend upon the use of Fourier series of ultradistributions with respect to a lattice. In the case of $\omega$-ultradistributions, to the best of our knowledge, this topic is not available in the literature. Section \ref{periodic ultradistributions} gives a concise presentation of the theory of Fourier series expansions of periodic ultradistributions.

Section \ref{Section F-L} is dedicated to discrete characterizations of the Fourier-Lebesgue wave front set of a non-quasianalytic $\omega$-ultradistribution with respect to a so-called $\omega$-moderate weight. In particular, our considerations apply to the wave front sets $WF_{(\omega)}$ and $WF_{\{\omega\}}$. It is worth mentioning that the latter two classes of wave front sets have been recently studied by Albanese et al. \cite{a-j-o} and Fern\'{a}ndez et al. \cite{f-g-j} in connection with regularity of solutions to linear PDE. Note that wave front sets with respect to Fourier-Lebesgue spaces were originally introduced in  \cite{j-p-t-t,Pilipovicthree,Pilipovictwo}, but we remark that here we allow weights with much larger growth than those considered in the forementioned works. In addition, our results apply to more general classes of ultradistributions. The authors believe that the Beurling-Bj\"{o}rck theory is the most natural framework for microlocal analysis based on Fourier-Lebesgue spaces. 

Our results from Section \ref{Section F-L} already include discrete characterizations of Gevrey $\{s\}$- and $(s)$-microregularity for $s>1$ \cite{Rodino}. Moreover, they also contain the case of microregularity with respect to certain classes of non-quasianalytic weight sequences. In Section \ref{section quasianalytic} we further extend our analysis to weight sequences satisfying milder assumptions. In the non-quasianalytic case, we shall show the following theorem. A lattice $\Lambda$ in $\mathbb{R}^{d}$ is simply a discrete subgroup of $\mathbb{R}^{d}$ which spans the real vector space $\mathbb{R}^{d}$. The dual lattice of $\Lambda$ is the discrete group $\Lambda^{\ast}=\{\mu^{\ast}\in\mathbb{R}^{d}:\: \mu\cdot \mu^{\ast}\in\mathbb{Z}, \:\forall \mu\in\Lambda\}$.

\begin{theorem} 
\label{th WF non-quasianalytic} Let $(M_p)_{p\in\mathbb{N}}$ be a weight sequence satisfying the conditions $(M.1)$, $(M.2)'$, and $(M.3)'$ and having associated function $M$ (see Subsection \ref{subsection sequences}). Suppose that  $\Lambda$ is a lattice in $\mathbb{R}^{d}$ and let $U$ be an open convex neighborhood of the origin such that $U\cap\Lambda^{*}=\{0\}$. The Roumieu wave front set $WF_{\{M_p\}}(f)$ $($the Beurling wave front set $WF_{(M_p)}(f)$$)$ of an ultradistribution $f\in{\mathcal{D}^{(M_p)}}' (\Omega)$ coincides with the complement in $\Omega\times(\mathbb{R}^{d}\setminus\{0\})$ of the set of all points $(x_0,\xi_0)$ for which 
there exist an open conic neighborhood $\Gamma$ of $\xi_0$ and $\varphi \in \mathcal{D}^{(M_p)}(\Omega\cap(x_0+U))$ with $\varphi\equiv1$ in a neighborhood of $x_0$  such that for some $r > 0$ $($for every $r>0$$)$
$$
\sup_{\mu \in \Gamma\cap \Lambda} |\widehat{\varphi f}(\mu)| e^{M(r\mu)} < \infty. 
$$
\end{theorem}

Our main result from Section \ref{section quasianalytic}, Theorem \ref{mainqa}, actually covers quasianalytic wave front sets. Specializing Theorem \ref{mainqa} to the analytic wave front set of a distribution, one obtains:

\begin{theorem} 
\label{th WF analytic} Let $\Lambda$ be a lattice in $\mathbb{R}^{d}$ and let $U$ be an open convex neighborhood of the origin such that $U\cap\Lambda^{*}=\{0\}$. The analytic wave front set $WF_{A}(f)$ of a distribution $f \in \mathcal{D}'(\Omega)$ coincides with the complement in $\Omega\times(\mathbb{R}^{d}\setminus\{0\})$ of the set of points $(x_0,\xi_0)$ for which there are an open neighborhood $V \subseteq \Omega\cap(x_0+U)$ of $x_0$,  an open conic neighborhood $\Gamma$ of $\xi_0$, and a bounded sequence $(f_p)_{p \in \mathbb{N}}$  in $\mathcal{E}'(\Omega\cap(x_0+U))$, with $f_p= f$ on $V$ for all $p \in \mathbb{N}$, such that for some $h > 0$
\[
  \sup_{\mu \in \Gamma \cap \Lambda} |\widehat{f_p} (\mu)| |\mu|^p  \leq h^{p+1}p!\:, \qquad \forall p \in \mathbb{N}. 
   \]

\end{theorem}
\section{Preliminaries}
\label{preliminaries}
In this section we fix the notation and explain the spaces of ultradifferentiable functions and ultradistributions needed in this article. 

\subsection{Spaces defined via weight functions}
\label{subsection weights} We start with the Beurling-Bj\"{o}rck approach to ultradistribution theory via weight functions \cite{Bjorck} (see also \cite{b-m-t}). A weight function on $\mathbb{R}^{d}$ is simply a non-negative measurable function. Throughout the article we shall always assume that $\omega$ is an even weight function satisfying $\omega(0)=0$ and the following three conditions:
\begin{itemize}
\item[$(\alpha)$]  $\omega(\xi_1 + \xi_2) \leq \omega(\xi_1) + \omega(\xi_2), \quad \forall \xi_1,\xi_2 \in \mathbb{R}^d,$
\item[$(\beta)$] $\displaystyle \int_{|\xi| > 1} \frac{\omega(\xi)}{|\xi|^{d + 1}} \mathrm{d}\xi < \infty,$
\item[$(\gamma)$] there exist $a \in \mathbb{R}$ and $C > 0$ such that $\omega(\xi) \geq a + C\log (1 + |\xi|), \quad \forall \xi \in \mathbb{R}^d.$
\end{itemize}
It should be noticed that measurability and subadditivity, namely, condition $(\alpha)$, ensure that $\omega$ is locally bounded \cite{beurling,hilleph}. 

Let $\Omega \subseteq \mathbb{R}^d$ be open, $K \Subset\Omega$ (a compact subset in $\Omega$) and $\lambda>0$. The Banach space $\mathcal{D}_{\omega}^\lambda(K)$ consists of those $\varphi  \in L^{1}(\mathbb{R}^d)$ such that $\operatorname*{supp} \varphi \subseteq K$ and
\[\| \varphi \|_\lambda= \| \varphi \|_{\omega,\lambda}:= \sup_{\xi \in \mathbb{R}^d} |\widehat{\varphi}(\xi)|e^{\lambda \omega(\xi)} < \infty.\]
Set further
\[ \mathcal{D}_{(\omega)}(\Omega) = \varinjlim_{K \Subset \Omega} \varprojlim_{\lambda \to \infty} \mathcal{D}_{\omega}^\lambda(K).\]
Condition $(\gamma)$ yields  $\mathcal{D}_{(\omega)}(\Omega)\subseteq \mathcal{D}(\Omega)$. Its dual $\mathcal{D}_{(\omega)}'(\Omega)$ is the ultradistribution space of class $(\omega)$ (or  Beurling type). We define $\mathcal{E}_{(\omega)}(\Omega)$ as the space of multipliers of $\mathcal{D}_{(\omega)}(\Omega)$, that is, a function  $\varphi\in\mathcal{E}_{(\omega)}(\Omega)$ if and only if $\varphi\psi \in\mathcal{D}_{(\omega)}(\Omega)$ for all $\psi \in \mathcal{D}_{(\omega)}(\Omega)$. Its topology is generated by the family of seminorms
$ \varphi \to  \| \varphi\psi \|_{\lambda}$, $ \lambda > 0,$ $\psi \in \mathcal{D}_{(\omega)}(\Omega). $ Clearly,
$\mathcal{E}'_{(\omega)}(\Omega)$ is the subspace of $\mathcal{D}_{(\omega)}'(\Omega)$ consisting of ultradistributions with compact support. The space $\mathcal{S}_{(\omega)}(\mathbb{R}^d)$ consists of all those $\varphi\in C^\infty(\mathbb{R}^d)$ such that
\[\sup_{x \in \mathbb{R}^d}|\varphi^{(\alpha)}(x)|e^{\lambda \omega(x)} < \infty \quad\mbox{and}\quad \sup_{\xi \in \mathbb{R}^d}|\widehat{\varphi}^{(\alpha)}(\xi)|e^{\lambda \omega(\xi)} < \infty, \quad \forall\lambda > 0,\: \forall\alpha\in\mathbb{N}^{d};\]
its Fr\'{e}chet space topology being defined in the canonical way. We shall fix the constants in the Fourier transform as
\[ \mathcal{F} \varphi (\xi) = \widehat{\varphi}(\xi) = \int_{\mathbb{R}^d} \varphi(x) e^{-2\pi i \xi \cdot x} \mathrm{d}x. \]
Since the Fourier transform is an automorphism of  $\mathcal{S}_{(\omega)}(\mathbb{R}^d)$, it can be extended by duality to $\mathcal{S}_{(\omega)}'(\mathbb{R}^d)$, the so-called space of $(\omega)$-tempered ultradistributions.  

Naturally, one may also define the spaces of class $\{\omega\}$  (Roumieu type)  $\mathcal{D}_{\{\omega\}}(\Omega)$, $\mathcal{D}'_{\{\omega\}}(\Omega)$, $\mathcal{E}_{\{\omega\}}(\Omega)$, $\mathcal{E}'_{\{\omega\}}(\Omega)$, $\mathcal{S}_{\{\omega\}}(\mathbb{R}^d)$, and $\mathcal{S}'_{\{\omega\}}(\mathbb{R}^d)$ by simply switching the universal quantifier $\forall$ over $\lambda$ to an existential one. 
When considering these spaces, we shall always assume that $\omega$ satisfies a stronger condition than $(\gamma)$, namely,
\begin{itemize}\item[$(\gamma_{0})$] $\displaystyle\lim_{|\xi|\to\infty} \frac{\omega(\xi)}{\log (1 + |\xi|)}=\infty$.
\end{itemize}
For instance, 
\[\mathcal{D}_{\{\omega\}}(\Omega) = \varinjlim_{K \Subset \Omega} \varinjlim_{\lambda \to 0^+} \mathcal{D}_{\omega}^\lambda(K),
\]
and $(\gamma_0)$ ensures that $\mathcal{D}_{\{\omega\}}(\Omega)\subsetneq \mathcal{D}(\Omega)$.

Note that if $\omega(\xi)=\log(1+|\xi|)$, one then recovers the classical Schwartz spaces as particular instances of the Beurling case. Another 
important example of a weight function is provided by the Gevrey weights $\omega(\xi) = |\xi|^{\frac{1}{s}},$  $s>1$; in such a case one obtains the well known Gevrey function and ultradistribution spaces \cite{Rodino,j-p-t-t}.

We shall also work with weighted Fourier-Lebesgue spaces \cite[Chap. II]{Bjorck}. A weight function $v$ is said to be $(\omega)$-moderate  ($\{\omega\}$-moderate) if there are $C, \lambda > 0$ (for every $\lambda > 0$ there is $C=C_{\lambda}>0$) such that
\begin{equation} 
v(\xi_1 + \xi_2) \leq Cv(\xi_1)e^{\lambda \omega(\xi_2)}, \quad \forall \xi_1,\xi_2 \in \mathbb{R}^d. 
\label{moderate}
\end{equation}
The classes of all $(\omega)$-moderate and $\{\omega\}$-moderate weight functions are denoted by $\mathscr{M}_{(\omega)}$ and $\mathscr{M}_{\{\omega\}}$, respectively. Let $q \in [1, \infty]$; the (weighted) Fourier--Lebesgue space $\mathcal{F} L^q_v$ (with respect to $v \in \mathscr{M}_{(\omega)}$) is the Banach space of all $f \in \mathcal{S}_{(\omega)}'(\mathbb{R}^d)$ such that $\widehat{f} \in L^q_v(\mathbb{R}^{d})$, that is, $\widehat{f}$ is locally integrable and  $\| f\|_{\mathcal{F} L^q_{v}}:=\| \widehat{f}\|_{L^q_{v}}=\| \widehat{f}\: v\|_{L^q} < \infty$. Clearly, if $v\in\mathscr{M}_{\{\omega\}}$, then $\mathcal{F} L^q_v\subset  \mathcal{S}_{\{\omega\}}'(\mathbb{R}^d)$.

\subsection{Weight sequences}
\label{subsection sequences} Another useful and widely used approach to the theory of ultradifferentiable functions and ultradistributions is via weight sequences \cite{Komatsu}. Let $(M_p)_{p \in \mathbb{N}}$ be a sequence of positive real numbers (with $M_0=1$). We will make use of some of the following conditions:\begin{itemize}
\item [$(M.1)\:$] $M^{2}_{p}\leq M_{p-1}M_{p+1},$  $p\geq 1$,
\item [$(M.2)'$]$M_{p+1}\leq A H^p M_p$, $p\in\mathbb{N}$, for some $A,H>0$,
\item [$(M.2)\:$] $ \displaystyle M_{p}\leq A H^p\min_{1\leq q\leq p} \{M_{q} M_{p-q}\},$ $p\in\mathbb{N}$, for some $A,H>0$,
\item [$(M.3)'$] $\displaystyle \sum_{p=1}^{\infty}\frac{M_{p-1}}{M_{p}}<\infty. $
\end{itemize}
The associated function of $(M_p)_{p}$ is defined as
 $$M(t):=\sup_{p\in\mathbb{N}}\log\frac{t^p}{M_p},\quad t>0.$$
Its log-convex regularization is the sequence $\displaystyle M^{c}_{p}= \sup_{t>0}\frac{t^{p}}{e^{M(t)}}$, $p\in\mathbb{N}$, the greatest log-convex minorant of $(M_p)_{p}$; note that $(M_p)_{p}$ satisfies (M.1) if and only if $(M_p)_{p}=(M_p^{c})_{p}$ (see \cite{Komatsu}).  As usual, the relation $M_p\subset N_p$ between two weight sequences means that there are $C,h>0$ such that 
$M_p\leq Ch^{p}N_{p},$ $p\in\mathbb{N}$. The stronger relation $M_p\prec N_p$ means that the latter inequality remains valid for every $h>0$ and a suitable $C=C_{h}>0$. 

Let $\Omega \subseteq \mathbb{R}^d$ be open. For $K \Subset \Omega$ and $h > 0$, one writes $\mathcal{E}^{\{M_p\},h}(K)$ for the space of all $\varphi \in C^\infty(\Omega)$ such that
\[ \| \varphi \|_{\mathcal{E}^{\{M_p\},h}(K)} := \sup_{\substack{x \in K \\ \alpha \in \mathbb{N}^d}} \frac{|\varphi^{(\alpha)}(x)|}{h^{|\alpha|}M_{|\alpha|}} < \infty, \]
and $\mathcal{D}^{\{M_p\},h}_{K}$ stands for the closed subspace of $\mathcal{E}^{\{M_p\},h}(K)$ consisting of functions with compact support in $K$. Further on,
\[ \mathcal{D}^{(M_p)}(\Omega) = \varinjlim_{K \Subset \Omega} \varprojlim_{h \rightarrow 0^+} \mathcal{D}^{\{M_p\},h}_{K}, \qquad \mathcal{D}^{\{M_p\}}(\Omega) = \varinjlim_{K \Subset \Omega} \varinjlim_{h \rightarrow \infty} \mathcal{D}^{\{M_p\},h}_{K}, \]

and
\[ \mathcal{E}^{(M_p)}(\Omega) = \varprojlim_{K \Subset \Omega} \varprojlim_{h \rightarrow 0^+} \mathcal{E}^{\{M_p\},h}(K), \qquad \mathcal{E}^{\{M_p\}}(\Omega) = \varprojlim_{K \Subset \Omega} \varinjlim_{h \rightarrow \infty} \mathcal{E}^{\{M_p\},h}(K);  \]
their duals are the spaces of $M_p$-ultradistributions and compactly supported $M_p$-ultradistributions of Beurling and Roumieu type, respectively \cite{Komatsu}. 

It is important to point out that under certain circumstances these spaces coincide with those discussed in Subsection \ref{subsection weights}. For instance, Petzsche and Vogt have shown \cite[Sect. 5]{Petzsche-V} (see also \cite[Satz 2.3]{Petzsche78}) that if the weight sequence $(M_p)_{p}$ satisfies $(M.1)$, $(M.2)$, $(M.3)'$, and the condition
\begin{itemize}
\item [$(M.4)\:$] $\displaystyle \frac{M_{p}}{p!}\subset \left(\frac{M_p}{p!}\right)^{c}$,
\end{itemize}
then one can always find a weight function $\omega$ fulfilling $(\alpha)$, $(\beta)$, $(\gamma_{0})$, and $\omega\asymp M$, 
such that $\mathcal{D}_{(\omega)}(\Omega)=\mathcal{D}^{(M_p)}(\Omega)$ and $\mathcal{D}_{\{\omega\}}(\Omega)=\mathcal{D}^{\{M_p\}}(\Omega)$, topologically. Furthermore, they proved, under $(M.2)$, that $(M.4)$ is equivalent to the so-called Rudin condition:
\begin{itemize}
\item [$(M.4)''$] $\displaystyle \max_{q\leq p}\left(\frac{M_{q}}{q!}\right)^{\frac{1}{q}}\leq A\left(\frac{M_{p}}{p!}\right)^{\frac{1}{p}},$ $p\in\mathbb{N}$, for some $A>0$.
\end{itemize}
($(M.4)''$ is equivalent to the property that $\mathcal{E}^{(M_p)}(\Omega)$ and $\mathcal{E}^{\{M_p\}}(\Omega)$ are inverse closed, cf. \cite{Rudin62,Rainer-S}.) Finally, it is worth mentioning that strong non-quasianalyticity (i.e., Komatsu's condition (M.3) \cite{Komatsu}) automatically yields $(M.4)$, as shown by Petzsche \cite[Prop. 1.1]{Petzsche88}.

\section{Periodic ultradistributions}
\label{periodic ultradistributions}
 Let $\Lambda$ be a lattice in $\mathbb{R}^{d}$. An ultradistribution $f$ is said to be $\Lambda$-periodic if $f(\:\cdot\:+\mu)=f$ for all $\mu\in\Lambda$. We denote as $\mathcal{D}'_{\mathfrak{p}_{\Lambda},(\omega)}$ and $\mathcal{D}'_{\mathfrak{p}_{\Lambda},\{\omega\}}$ the spaces of $\Lambda$-periodic $\omega$-ultradistributions of Beurling and Roumieu type, respectively. We also consider $\mathcal{D}_{\mathfrak{p}_{\Lambda},(\omega)}=\mathcal{E}_{(\omega)}(\mathbb{R}^{d})\cap \mathcal{D}'_{\mathfrak{p}_{\Lambda},(\omega)}$ and $\mathcal{D}_{\mathfrak{p}_{\Lambda},\{\omega\}}=\mathcal{E}_{\{\omega\}}(\mathbb{R}^{d})\cap \mathcal{D}'_{\mathfrak{p}_{\Lambda},\{\omega\}}$, which are easily seen to be closed subspaces of $\mathcal{E}_{(\omega)}(\mathbb{R}^{d})$ and $\mathcal{E}_{\{\omega\}}(\mathbb{R}^{d})$, respectively. 

We shall show in this section that every $\Lambda$-periodic $\omega$-ultradistribution can be expanded into a Fourier series. For it, we select a \emph{fundamental region} $I_{\Lambda}$ for the lattice, namely, a connected set $I_{\Lambda}\subseteq \mathbb{R}^{d}$ with the property that the restriction of the quotient mapping $\mathbb{R}^{d}\to \mathbb{R}^{d}/\Lambda$ to $I_{\Lambda}$  is a bijection. The set  $I_{\Lambda}$ is of course a $d$-dimensional parallelepiped and $\mathbb{R}^{d}=\bigcup_{\mu\in\Lambda}(\mu+I_\Lambda)$. It can be readily shown that $|\Lambda|:=\operatorname*{vol}(I_{\Lambda})$ does not depend on the choice of the fundamental region $I_{\Lambda}$. As in the introduction, $\Lambda^{\ast}$ stands for the dual lattice of $\Lambda$. Clearly, $\Lambda=\Lambda^{\ast\ast}$. Many of our arguments in this section are based on the Poisson summation formula \cite{Hormander}, which in this context takes the form
\begin{equation}
\label{Poisson}
\sum_{\mu\in\Lambda}e^{2\pi i\xi \cdot \mu}=\frac{1}{|\Lambda|}\sum_{\mu^*\in\Lambda^*}\delta(\xi-\mu^*) \quad \mbox{in }\mathcal{S}'(\mathbb{R}^d).
\end{equation}
Given $f\in\mathcal{E}'_{(\omega)}(\mathbb{R}^{d})$ (or whenever it makes sense), we denote as $f_{\mathfrak{p}_{\Lambda}}$ its \emph{$\Lambda$-periodization}, that is, the $\Lambda$-periodic $\omega$-ultradistribution
\begin{equation}
\label{eqperiodization} 
f_{\mathfrak{p}_{\Lambda}}:=\sum_{\mu\in\Lambda}f(\:\cdot\:+\mu).
\end{equation}
We begin with a useful lemma. 

\begin{lemma}\label{lemma partition unity}
There is $\eta\in\mathcal{D}_{(\omega)}(\mathbb{R}^{d})$ such that $\eta_{\mathfrak{p}_{\Lambda}}=1$, namely,
\begin{equation}
\label{eqpu}
\sum_{\mu\in\Lambda}\eta(x+\mu)=1, \quad x\in\mathbb{R}^{d}.
\end{equation}
\end{lemma}
\begin{proof} Fourier transforming (\ref{eqpu}) and employing the Poisson summation formula (\ref{Poisson}), one obtains that the relation (\ref{eqpu}) would be satisfied if we find $\eta\in\mathcal{D}_{(\omega)}(\mathbb{R}^{d})$ such that $\widehat{\eta}(0)=|\Lambda|$ and $\widehat{\eta}(\mu^*)=0$ for every $\mu^*\in\Lambda^*\setminus\{0\}$. Select $d$ vectors $\mu_{1},\dots,\mu_{d}$ that generate the Abelian group $\Lambda$ and that span $\mathbb{R}^{d}$, and pick $\varphi\in\mathcal{D}_{(\omega)}(\mathbb{R}^{d})$ such that $\int_{\mathbb{R}^{d}}\varphi(x)\mathrm{d}x=|\Lambda|$. The function $\eta$ given in Fourier side as 
$$\widehat{\eta}(\xi)=\widehat{\varphi}(\xi)\prod_{j=1}^{d}\frac{\sin(2\pi\xi \cdot  \mu_j)}{2\pi \xi \cdot  \mu_j}$$
 satisfies all requirements.
\end{proof}
 
We employ the notation $e_y(x) = e^{2\pi i y \cdot x}$, where $y$ is a fixed vector of $\mathbb{R}^d$. Given a locally integrable $\Lambda$-periodic function $\phi$, its Fourier coefficients with respect to the lattice $\Lambda$ are given by 
\begin{equation}
\label{eqfcoeff} c_{\mu^*}=c_{\mu^*}(\phi)=\frac{1}{|\Lambda|}\int_{I_{\Lambda}}\phi(x)e_{-\mu^*}(x)\mathrm{d}x, \quad \mu^*\in\Lambda^*.
\end{equation}
Obviously, the integral in (\ref{eqfcoeff}) does not depend on the choice of $I_{\Lambda}$; in fact, these coefficients can be also computed as 
\begin{equation}
\label{eqfcoeff2} c_{\mu^*}=\frac{1}{|\Lambda|}\int_{\mathbb{R}^{d}}\phi(x)\eta(x)e_{-\mu^*}(x)\mathrm{d}x=\frac{1}{|\Lambda|}\widehat{(\phi\eta)}(\mu^*),
 \end{equation}
where $\eta$ is as in Lemma \ref{lemma partition unity}. Note that we thus have $c_{\mu^*}(e_{\nu^*})=\delta_{\nu^{*}\mu^{*}}$, as  immediately follows from (\ref{eqfcoeff2}). The next lemma shows that $\mathcal{D}_{\mathfrak{p}_{\Lambda},(\omega)}$ and $\mathcal{D}_{\mathfrak{p}_{\Lambda},\{\omega\}}$ are isomorphic to t.v.s of $\omega$-rapidly decreasing functions on the dual lattice $\Lambda^{\ast}$, namely, 
$$\mathcal{S}_{(\omega)}(\Lambda^{\ast})=\varprojlim_{\lambda\to\infty} \mathcal{S}^{\lambda}_{\omega}(\Lambda^{\ast})\quad\mbox{and}\quad \mathcal{S}_{\{\omega\}}(\Lambda^{\ast})=\varinjlim_{\lambda\to0^{+}} \mathcal{S}^{\lambda}_{\omega}(\Lambda^{\ast}),$$
where 
$$\displaystyle\mathcal{S}^{\lambda}_{\omega}(\Lambda^{\ast})=\left\{(c_{\mu^*})_{\mu^*\in\Lambda^*}\in \mathbb{C}^{\Lambda^{\ast}}:\: \sigma_{\lambda}((c_{\mu^*})):= \sup_{\mu^*\in\Lambda^*}|c_{\mu^*}|e^{\lambda \omega(\mu^*)}<\infty \right\}, \quad \lambda\in\mathbb{R}.$$
(In the Roumieu case we assume that $(\gamma_{0})$ holds.)
\begin{lemma}\label{ultrafunc}
If $\phi \in\mathcal{D}_{\mathfrak{p}_{\Lambda},(\omega)}$ $($$\phi \in\mathcal{D}_{\mathfrak{p}_{\Lambda},\{\omega\}}$$)$, then
\begin{equation}
\phi= \sum_{\mu^* \in \Lambda^\ast} c_{\mu^*}e_{\mu^*},
\label{series}
\end{equation}
with convergence in $\mathcal{D}_{\mathfrak{p}_{\Lambda},(\omega)}$ $($$\mathcal{D}_{\mathfrak{p}_{\Lambda},\{\omega\}}$$)$,
where the Fourier coefficients $c_{\mu^*}$ are given by $(\ref{eqfcoeff})$. Moreover, the mapping $\phi\mapsto (c_{\mu^*})_{\mu^*\in\Lambda^*}$ yields the $($t.v.s.$)$ isomorphisms $\mathcal{D}_{\mathfrak{p}_{\Lambda},(\omega)} \cong \mathcal{S}_{(\omega)}(\Lambda^{\ast})$ and $\mathcal{D}_{\mathfrak{p}_{\Lambda},\{\omega\}} \cong \mathcal{S}_{\{\omega\}}(\Lambda^{\ast})$.
\end{lemma}
\begin{proof} Let $\eta\in\mathcal{D}_{(\omega)}(\mathbb{R}^{d})$ be as in Lemma \ref{lemma partition unity}. The relation (\ref{eqfcoeff2}) then yields
$$
 |c_{\mu^*}(\phi)| =|\Lambda|^{-1} \left|\widehat{\eta \phi}(\mu^*)\right| \leq |\Lambda|^{-1}\|\eta \phi\|_\lambda e^{-\lambda\omega(\mu^*)},
 $$
which shows the continuity of $\phi\mapsto (c_{\mu^*}(\phi))_{\mu^*\in\Lambda^*}$ in both cases. On the other hand, if $(c_{\mu^*})_{\mu^*\in\Lambda^*}\in \mathcal{S}_{(\omega)}(\Lambda^{\ast})$ (or $(c_{\mu^*})_{\mu^*\in\Lambda^*}\in \mathcal{S}_{\{\omega\}}(\Lambda^{\ast})$), $F$ is a finite subset of $\Lambda^*$, and $\psi\in \mathcal{D}_{(\omega)}(\mathbb{R}^{d})$ ($\psi\in \mathcal{D}_{\{\omega\}}(\mathbb{R}^{d})$), then 
$$
\| \psi \sum_{\mu^*\in\Lambda^*\setminus F  }c_{\mu^*}e_{\mu^*}\|_{\lambda}\leq \|\psi\|_{\lambda} \sigma_{2\lambda}((c_{\mu^*}))\sum_{\mu^*\in\Lambda^*\setminus F  } e^{-\lambda \omega(\mu^*)}\:.
$$
This proves that $\sum_{\mu^*\in\Lambda^* }c_{\mu^*}e_{\mu^*}$ is summable in $\mathcal{D}_{\mathfrak{p}_{\Lambda},(\omega)}$ (in $\mathcal{D}_{\mathfrak{p}_{\Lambda},\{\omega\}}$) and that the mapping $(c_{\mu^*})_{\mu^*\in\Lambda^*}\mapsto \sum_{\mu^*\in\Lambda^*}c_{\mu^*}e_{\mu*}$ is also continuous. The convergence of (\ref{series}) is now a consequence of the injectivity of $\phi\mapsto (c_{\mu^*}(\phi))_{\mu^*\in\Lambda^*}$, which of course follows from the case $\Lambda=\mathbb{Z}^{d}$ by a linear change of variables\footnote{Every lattice in $\mathbb{R}^{d}$ is of the form $\Lambda=T(\mathbb{Z}^{d})$, where $T$ is an invertible matrix.}. Alternatively, this injectivity can also be established as follows. If  $\phi$ is such that $c_{\mu^*}(\phi)=|\Lambda|^{-1}\widehat{(\phi \eta)}(\mu^{*})=0$ for every $\mu^*\in\Lambda^*$, then, for an arbitrary $\psi\in\mathcal{D}(\mathbb{R}^{d})$, we have $\int_{\mathbb{R}^{d}}\phi(x)\psi(x)\mathrm{d}x =\int_{\mathbb{R}^{d}}\phi(x)\eta(x)\psi_{\mathfrak{p}_{\Lambda}}(x)
 \mathrm{d}x$. Hence, by the Poisson summation formula, 
\begin{align*}
\int_{\mathbb{R}^{d}}\phi(x)\psi(x)\mathrm{d}x&=\frac{1}{|\Lambda|}\left\langle \widehat{(\phi \eta)}(-\xi),\sum_{\mu^*\in\Lambda^*}\widehat{\psi}(\xi)\delta(\xi-\mu^*)\right\rangle
\\
&
=\frac{1}{|\Lambda|}\sum_{\mu^*\in\Lambda^*}\widehat{\psi}(-\mu^*)\widehat{(\phi \eta)}(\mu^{*})=0.
\end{align*}
Therefore, we must have $\phi=0$.
\end{proof}

We are ready to deal with Fourier expansions of $\omega$-ultradistributions. The treatment is similar to the distribution case \cite{zeeman}, but we give the details for the sake of completeness. Observe first that a standard argument shows that $$\mathcal{S}'_{(\omega)}(\Lambda^{\ast})= \varinjlim_{\lambda\to\infty} \mathcal{S}^{-\lambda}_{\omega}(\Lambda^{\ast})\quad\mbox{and}\quad \mathcal{S}'_{\{\omega\}}(\Lambda^{\ast})=\varprojlim_{\lambda\to0^{+}}\mathcal{S}^{-\lambda}_{\omega}(\Lambda^{\ast}).$$
Lemma \ref{ultrafunc} then yields that the duals of $\mathcal{D}_{\mathfrak{p}_{\Lambda},(\omega)}$ and $\mathcal{D}_{\mathfrak{p}_{\Lambda},\{\omega\}}$ are isomorphic to these spaces and their elements can be expanded as $g=\sum_{\mu^* \in \Lambda^{\ast}}c_{\mu^*}(g)e_{\mu^\ast}$ with $( c_{\mu^*}(g))_{\mu^*\in\Lambda^*}$ being an element of $\mathcal{S}'_{(\omega)}(\Lambda^{\ast})$ or $\mathcal{S}'_{\{\omega\}}(\Lambda^{\ast})$, respectively, and 
$$
\langle g,\phi \rangle_{(\mathcal{D}_{\mathfrak{p}_{\Lambda},\omega})'\times \mathcal{D}_{\mathfrak{p}_{\Lambda},\omega}}=|\Lambda| \sum_{\mu^*\in\Lambda^*} c_{-\mu^*}(g)c_{\mu^*}(\phi),
$$
where $\omega$ stands for either the Beurling case $(\omega)$ or the Roumieu case $\{\omega\}$, respectively.

In the rest of the discussion $\eta$ stands for a test function as in Lemma \ref{lemma partition unity}. Given 
$f\in \mathcal{D}'_{\mathfrak{p}_{\Lambda},(\omega)}$ ($f\in\mathcal{D}'_{\mathfrak{p}_{\Lambda},\{\omega\}}$), we identify it with an element $g$ of the dual of  $\mathcal{D}_{\mathfrak{p}_{\Lambda},(\omega)}$ ($\mathcal{D}_{\mathfrak{p}_{\Lambda},\{\omega\}}$) as follows:
\begin{equation}
\label{eqidentificationperiodic}
\langle g,\phi \rangle_{(\mathcal{D}_{\mathfrak{p}_{\Lambda},\omega})'\times \mathcal{D}_{\mathfrak{p}_{\Lambda},\omega}}:=\langle f, \eta \phi \rangle_{\mathcal{D}'_{\omega}(\mathbb{R}^{d})\times \mathcal{D}_{\omega}(\mathbb{R}^{d})}, \quad \phi\in \mathcal{D}_{\mathfrak{p}_{\Lambda},\omega}\:.
\end{equation}
The definition of $g$ is independent of the choice of $\eta$, as can be readily verified. The $\Lambda$-periodic $\omega$-ultradistribution $f$ can be recovered on $\mathbb{R}^{d}$ from $g$ as 
\begin{equation}
\label{eqidentificationperiodic2}
\langle f,\varphi\rangle_{\mathcal{D}'_{\omega}(\mathbb{R}^{d})\times \mathcal{D}_{\omega}(\mathbb{R}^{d})}=\langle g,\varphi_{\mathfrak{p}_{\Lambda}} \rangle_{(\mathcal{D}_{\mathfrak{p}_{\Lambda},\omega})'\times \mathcal{D}_{\mathfrak{p}_{\Lambda,\omega}}}\:, \quad \varphi\in  \mathcal{D}_{\omega}(\mathbb{R}^{d}).
\end{equation}
Using $g$, one defines the Fourier coefficients of $f$ as
\begin{equation}
\label{eqfcoeff3}
c_{\mu^*}(f):=c_{\mu^*}(g)=\frac{1}{|\Lambda|}\langle f, \eta e_{-\mu^*} \rangle, \quad \mu^*\in\Lambda^*.
\end{equation} 
Summarizing, we have:
\begin{proposition}\label{perdis} Every $f\in \mathcal{D}'_{\mathfrak{p}_{\Lambda},(\omega)}$ $($$f\in\mathcal{D}'_{\mathfrak{p}_{\Lambda},\{\omega\}}$$)$ can be expanded as
\begin{equation}
\label{series2}
f=\sum_{\mu^* \in \Lambda^*} c_{\mu^*}e_{\mu^*} \quad \mbox{in } \mathcal{D}'_{(\omega)}(\mathbb{R}^{d}) \quad (\mbox{in } \mathcal{D}'_{\{\omega\}}(\mathbb{R}^{d})), 
\end{equation}
where the Fourier coefficients $c_{\mu^*}$ are given by $(\ref{eqfcoeff3})$. The mapping $f\mapsto (c_{\mu^*})_{\mu^*\in\Lambda^*}$ provides the $($t.v.s.$)$ isomorphisms $\mathcal{D}'_{\mathfrak{p}_{\Lambda},(\omega)} \cong \mathcal{S}'_{(\omega)}(\Lambda^{\ast})$ and $\mathcal{D}'_{\mathfrak{p}_{\Lambda},\{\omega\}} \cong \mathcal{S}'_{\{\omega\}}(\Lambda^{\ast})$.
\end{proposition}
\begin{proof} Writing  $g=\sum_{\mu^* \in \Lambda^*} c_{\mu^*}(f)e_{\mu^*}$ for the ultradistribution given by (\ref{eqidentificationperiodic}), we obtain 
$$\langle f,\varphi\rangle=\langle g,\varphi_{\mathfrak{p}_{\Lambda}} \rangle=|\Lambda|\sum_{\mu^* \in \Lambda^*} c_{\mu^*}(f) c_{-\mu^*}(\varphi_{\mathfrak{p}_{\Lambda^*}} )=\sum_{\mu^* \in \Lambda^*} c_{\mu^*}(f)\int_{\mathbb{R}^{d}}\eta(x)\varphi_{\mathfrak{p}_{\Lambda}}(x)e_{\mu^*}(x) \mathrm{d}x, 
$$ which proves (\ref{series2}) because $\int_{\mathbb{R}^{d}}\varphi(x)e_{\mu^*}(x)\mathrm{d}x =\int_{\mathbb{R}^{d}}\eta(x)\varphi_{\mathfrak{p}_{\Lambda}}(x)e_{\mu^*}(x)\mathrm{d}x $. Note that the correspondence $f\mapsto g$ provides continuous mappings $\mathcal{D}'_{\mathfrak{p}_{\Lambda},(\omega)}\to (\mathcal{D}_{\mathfrak{p}_{\Lambda},(\omega)})'$ and $\mathcal{D}'_{\mathfrak{p}_{\Lambda},\{\omega\}}\to (\mathcal{D}_{\mathfrak{p}_{\Lambda},\{\omega\}})'$,  as directly follows from (\ref{eqidentificationperiodic}). Their inverse mappings (cf. (\ref{eqidentificationperiodic2})) are the transposes of the continuous mappings $\mathcal{D}_{(\omega)}(\mathbb{R}^{d})\to \mathcal{D}_{\mathfrak{p}_{\Lambda},(\omega)}$ and $\mathcal{D}_{\{\omega\}}(\mathbb{R}^{d})\to \mathcal{D}_{\mathfrak{p}_{\Lambda},\{\omega\}}$ given by $\varphi\mapsto\varphi_{\mathfrak{p}_{\Lambda}}$, so they are continuous as well.
\end{proof}

In view of Proposition \ref{perdis}, we can canonically identify  $\mathcal{D}'_{\mathfrak{p}_{\Lambda},(\omega)} $ and $\mathcal{D}'_{\mathfrak{p}_{\Lambda},\{\omega\}}$ with the duals of $\mathcal{D}_{\mathfrak{p}_{\Lambda},(\omega)}$ and $\mathcal{D}_{\mathfrak{p}_{\Lambda},\{\omega\}}$ via the Fourier series (\ref{series2}), so that we simply write $\mathcal{D}'_{\mathfrak{p}_{\Lambda},(\omega)}=(\mathcal{D}_{\mathfrak{p}_{\Lambda},(\omega)})' $ and $\mathcal{D}'_{\mathfrak{p}_{\Lambda},\{\omega\}}=(\mathcal{D}_{\mathfrak{p}_{\Lambda},\{\omega\}})'$. This convention of course amounts to the same as the identification $f=g$ by means of (\ref{eqidentificationperiodic}) and 
(\ref{eqidentificationperiodic2}).

We can also define weighted Fourier-Lebesgue spaces with respect to the dual lattice $\Lambda^{*}$. Let $v \in \mathscr{M}_{(\omega)}$ and $q \in[1, \infty]$. The Banach space $\mathcal{F} l^q_{v,\Lambda^{*}}$ consists of all $f =\sum_{\mu^*\in\Lambda^*} c_{\mu^*}e_{\mu^*}\in \mathcal{D}'_{\mathfrak{p}_{\Lambda},(\omega)}$ such that $(c_{\mu^*})_{\mu^* \in \Lambda^*} \in l^q_v(\Lambda^*)$, namely,  
$$
|f|_{\mathcal{F} l^{q}_{v,\Lambda^*}} := \|(c_{\mu^*})_{\mu^*\in\Lambda^*}\|_{l^q_v(\Lambda^*)}=\|(c_{\mu^*}v(\mu^*))_{\mu^*\in\Lambda^*}\|_{l^q(\Lambda^*)} < \infty.
$$

\begin{remark}
\label{rFSu} Analogous results hold for $M_p$-ultradistributions under the assumptions $(M.1)$, $(M.2)'$, and $(M.3)'$. In fact, assume that the sequence $(M_{p})_{p\in\mathbb{N}}$ satisfies these three conditions  and consider ${\mathcal{D}^{(M_p)}_{\mathfrak{p}_{\Lambda}}}'$ and ${\mathcal{D}^{\{M_p\}}_{\mathfrak{p}_{\Lambda}}}'$, the subspaces of ${\mathcal{D}^{(M_p)}}'(\mathbb{R}^d)$ and ${\mathcal{D}^{\{M_p\}}}'(\mathbb{R}^d)$, respectively, consisting of $\Lambda$-periodic $M_p$-ultradistributions. Then, every $\Lambda$-periodic $M_p$-ultradistribution admits the Fourier expansion (\ref{series2}) with convergence in ${\mathcal{D}^{(M_p)}}'(\mathbb{R}^d)$ or ${\mathcal{D}^{\{M_p\}}}'(\mathbb{R}^d)$, respectively. Furthermore,  $f\mapsto (c_{\mu^*})_{\mu^*\in\Lambda^*}$ also yields t.v.s. isomorphisms $\displaystyle{\mathcal{D}^{(M_p)}_{\mathfrak{p}_{\Lambda}}}' \cong {\mathcal{S}^{(M_p)}}'(\Lambda^{\ast}):= \varinjlim_{h\to\infty} \mathcal{S}^{M_p,-h}(\Lambda^{\ast})$ and $\displaystyle{\mathcal{D}^{\{
 M_p\}}_{
 \mathfrak{p}_{\Lambda}}}' \cong {\mathcal{S}^{\{M_p\}}}'(\Lambda^{\ast}):=\varprojlim_{h\to0^{+}}\mathcal{S}^{M_p,-h}(\Lambda^{\ast})$, where 
$$
\mathcal{S}^{M_p,-h}(\Lambda^{\ast})=\left\{(c_{\mu^*})_{\mu^*\in\Lambda^*}\in \mathbb{C}^{\Lambda^{\ast}}:\: \sup_{\mu^*\in\Lambda^*}|c_{\mu^*}|e^{- M(h\mu^*)}<\infty\right\}.
$$ 
The proofs of these assertions can be obtained exactly as for $\omega$-ultradistributions. We also refer to \cite{Gorba,Petzsche78} for studies involving Fourier series of $M_p$-ultradistributions.
 
\end{remark}
\smallskip

\section{Wave front sets of Fourier-Lebesgue type}
\label{Section F-L}
The aim of this section is to provide a discrete characterization of wave front sets of Fourier-Lebesgue type. Besides the conditions $(\alpha)$, $(\beta)$, and $(\gamma)$, we impose throughout this section the following additional assumption on $\omega$: The weight $\omega$ is a non-decreasing 
function of $|\xi|$, namely,
\begin{itemize}
\item [$(\alpha_0)$] $\omega(\xi)=\omega_{0}(|\xi|)$, $\xi\in\mathbb{R}^{d}$, where $\omega_{0}:[0,\infty)\to[0,\infty)$ is non-decreasing. 
\end{itemize}

We need to introduce some notation in order to define the Fourier-Lebesgue wave front set of an $\omega$-ultradistribution. 
 Let $v\in\mathscr{M}_{(\omega)}$, let $q\in[1,\infty]$, and let $\Gamma$ be a cone in $\mathbb{R}^d$. If $g\in\mathcal{S}'_{(\omega)}(\mathbb{R}^d)$ is such that $\widehat{g}$ is locally integrable in an open neighborhood of $\Gamma$, we consider the seminorm
\begin{equation}
\label{eqsFL}
|g|_{\mathcal{F} L^{q,\Gamma}_v} :=\|\widehat{g}\|_{L^q_{v}(\Gamma)}=\|\widehat{g}\:1_\Gamma \|_{L^q_{v}},
\end{equation}
where $1_A$ stands for the characteristic function of a set $A$. The seminorm (\ref{eqsFL}) is in particular well-defined if $g\in\mathcal{E}'_{(\omega)}(\mathbb{R}^d)$, but naturally it might become $\infty$.

Let now $f\in\mathcal{D}'_{(\omega)}(\Omega)$. The $\omega$-ultradistribution $f$ is said to be \emph{$\mathcal{F} L^q_v$-microlocally regular} at the point $(x_0,\xi_0)\in \Omega\times (\mathbb{R}^d \setminus \{ 0 \})$ if there are an open conic neighborhood $\Gamma$ of $\xi_0$ and a test function $\varphi\in\mathcal{D}_{(\omega)}(\Omega)$ with $\varphi(x_0)\neq0$ such that 
\begin{equation}
\label{eqe1}
|\varphi f|_{\mathcal{F} L^{q,\Gamma}_v }<\infty.
\end{equation}
The wave front set $WF_{\mathcal{F} L^q_v}(f)$ consists of all those points $(x_0,\xi_0)\in \Omega\times (\mathbb{R}^d \setminus \{ 0 \})$ such that $f$ is \emph{not} $\mathcal{F} L^q_v$-microlocally regular at  $(x_0,\xi_0)$.

 We are now ready to state the main theorem of this section, a discrete characterization of $WF_{\mathcal{F} L^q_v}(f)$ with respect to a lattice.

\begin{theorem}\label{main} Let $\Lambda$ be a lattice in $\mathbb{R}^{d}$, $f \in \mathcal{D}_{(\omega)}'(\Omega)$, $v\in\mathscr{M}_{(\omega)}$, $q\in[1,\infty]$, and $(x_0, \xi_0) \in \Omega \times (\mathbb{R}^d \backslash \{ 0 \})$. Suppose that $U$ is an  open convex neighborhood of the origin such that $U\cap\Lambda^{*}=\{0\}$ and $x_0+U\subseteq\Omega$. Then, the following statements are equivalent:
\begin{itemize}
\item[$(i)$] There are an open conic neighborhood $\Gamma$ of $\xi_0$ and $\varphi \in \mathcal{D}_{(\omega)}(x_{0}+U)$ with $\varphi(x_0)\neq 0$ such that
\[
 \left\|(\widehat{\varphi f}(\mu))_{\mu\in\Gamma\cap\Lambda} \right\|_{l^{q}_{v}(\Gamma\cap\Lambda)} < \infty.
\]\item[$(ii)$] $f$ is $\mathcal{F} L^q_v$-microlocally regular at  $(x_0,\xi_0)$.
\end{itemize}
\end{theorem}
Part of the proof of Theorem \ref{main} is based on the ensuing lemma:

\begin{lemma}\label{mainlemma} $\mbox{ }$
\begin{itemize}
\item[$(i)'$] Condition $(i)$ from Theorem \ref{main} implies that there are an open conic neighborhood $\Gamma_1$ of $\xi_0$ and an open neighborhood $U_1\subseteq U$ of the origin such that for every bounded set $B \subseteq \mathcal{D}_{(\omega)}(x_0+U_1)$
\[ \sup_{\psi \in B} \left\|(\widehat{\psi f}(\mu))_{\mu\in \Gamma_1\cap\Lambda} \right\|_{l^{q}_{v}(\Gamma_1\cap\Lambda)}<\infty .\]
\item[$(ii)'$] Condition $(ii)$ from Theorem \ref{main} implies that there are an open conic neighborhood $\Gamma_1$ of $\xi_0$ and an open neighborhood $U_1$ of the origin such that for every bounded set $B \subseteq \mathcal{D}_{(\omega)}(x_0+U_1)$
\[ \sup_{\psi \in B} |\psi f|_{ \mathcal{F}L^{q,\Gamma_1}_{v}} 
< \infty.\]
\end{itemize}
\end{lemma}
\begin{proof}
We only prove $(i)'$, because the second assertion can be established in a similar fashion by replacing sums by integrals. Assume condition $(i)$ of Theorem \ref{main}. We begin by finding suitable $\Gamma_1$ and $U_1$.
Choose an open conic neighborhood $\Gamma_1$ of $\xi_0$ such that $\overline{\Gamma}_1 \subseteq \Gamma \cup \{ 0 \}$. Let $m \in \mathbb{Z}_+$  be such that $m^{-1}$ is smaller than the distance between $\partial \Gamma$ and the intersection of $\Gamma_1$ with the unit sphere, and also smaller than
the distance between  $\partial \Gamma_1$  and the intersection of $\overline{\mathbb{R}^d \backslash \Gamma}$ with the unit
sphere. Hence $\xi \in \Gamma_1$ and $y \notin \Gamma$ imply that $|\xi - y| \geq  m^{-1}\max(|\xi|, |y|)$. As $U_1$ we select any open neighborhood $U_1$ of $0$ such that $\overline{U}_1\Subset U\cap\{x:\:\varphi(x+x_0)\neq 0\}=:U_2$.  In addition, pick $\kappa\in\mathcal{D}_{(\omega)}(x_0+U_2)$ such that $\kappa\equiv 1$ on $x_0+U_1$.  Observe that\footnote{In fact, Beurling theorem \cite{beurling,Bjorck} tells us that the condition $(\beta)$ is equivalent to the regularity of the Beurling algebras $L^{1}_{e^{\omega_{\lambda}}}$ with weights $\omega_{\lambda}(\xi):=\lambda\omega(\xi)$. As a consequence of the general theory of regular commutative Banach algebras \cite{G-R-S}, one obtains that analytic functions act locally on each $\mathcal{F}(L^{1}_{e^{\omega_{\lambda}}})$; in particular, $\mathcal{E}_{(\omega)}(\mathbb{R}^{d})$ and $\mathcal{E}_{\{\omega\}}(\mathbb{R}^{d})$ are inverse closed.} $\theta:=\kappa/\varphi \in\mathcal{D}_{(\omega)}(x_0+U_2)$. Let now 
 $B\subseteq \mathcal{D}_
 {(\omega)}(x_0+U_1)$ be a bounded subset. If $\psi\in B$ we have that $\psi f=\varphi \psi_1 f$, where $\psi_1$ is an element of the bounded set $B_1:=\theta B\subseteq \mathcal{D}_{(\omega)}(x_0+U_1)$.  Next, notice that there is a fundamental region $I_{\Lambda^{\ast}}$ of the dual lattice $\Lambda^*$ such that $x_0+U\subseteq I_{\Lambda^*}$. Since for an arbitrary $g\in \mathcal{E}'_{(\omega)}(\mathbb{R}^{d})$ having support in  the interior of $I_{\Lambda^{\ast}}$ we have that the Fourier coefficients of its $\Lambda^*$-periodization (cf. Section \ref{periodic ultradistributions}) are given by $(\widehat{g}(\mu))_{\mu\in\Lambda}$, we obtain
\[ (\varphi f)_{p_{\Lambda^*}} = \sum_{\mu \in \Lambda} a_\mu e_\mu, \] 
with $a_\mu=\widehat{\varphi f}(\mu)$. Furthermore, since $(\psi f)_{\mathfrak{p}_{\Lambda^*}}= (\varphi f)_{\mathfrak{p}_{\Lambda^*}}(\psi_1)_{\mathfrak{p}_{\Lambda^*}}$, we conclude that
\[ \widehat{\psi f}(\mu) = \sum_{\beta \in \Lambda} a_\beta \widehat{\psi}_1(\mu-\beta), \qquad \mu\in\Lambda.  \]
Thus, by (\ref{moderate}),
\begin{align*} 
 \left\|(\widehat{\psi f}(\mu))_{\mu\in\Gamma_1\cap \Lambda} \right\|_{l^{q}_{v}(\Gamma_1\cap\Lambda)} 
&\leq C  \left \|\left(\sum_{\beta\in\Lambda}|a_\beta|v(\beta)|\widehat{\psi_1}(\mu-\beta)|e^{\lambda \omega(\mu-\beta)} \right)_{\mu\in\Gamma_1\cap\Lambda}\right \|_{l^q(\Gamma_1\cap\Lambda)}
 \\
& \leq C(I_1(\psi_1) + I_2(\psi_1)), \quad \forall\psi\in B,
\end{align*}
where
\[ I_1(\psi_1) = \left \|\left(\sum_{\beta  \in \Gamma \cap \Lambda}|a_\beta|v(\beta)|\widehat{\psi_1}(\mu-\beta)|e^{\lambda \omega(\mu-\beta)}\right)_{\mu\in\Gamma_1\cap\Lambda} \right \|_{l^q(\Gamma_1\cap\Lambda)},
\]
and
\[ I_2(\psi_1) = \left \|\left(\sum_{\beta  \notin \Gamma \cap \Lambda}|a_\beta|v(\beta)|\widehat{\psi_1}(\mu-\beta)|e^{\lambda \omega(\mu-\beta)} \right)_{\mu\in\Gamma_1\cap\Lambda}\right \|_{l^q(\Gamma_1\cap\Lambda)}
.\]
Young's inequality and the boundedness of $B_1$ imply that
\[ \sup_{\psi_1 \in B_1}  I_1(\psi_1) \leq  \left\|(\widehat{\varphi f}(\mu))_{\mu\in\Gamma\cap \Lambda} \right\|_{l^{q}_{v}(\Gamma\cap\Lambda)} \sup_{\psi_1 \in B_1} \sum_{\beta \in \Lambda} |\widehat{\psi_1}(\beta)|e^{\lambda\omega(\beta)} < \infty. \]
We now estimate $I_2(\psi_1)$. By Proposition \ref{perdis} and the fact that $v$ is $\omega$-moderate there exist $D, \lambda_0 > 0$ such that
\begin{equation}
\label{eqe2}
 |a_\beta|v(\beta) \leq De^{\lambda_0\omega(\beta)}, \qquad \forall \beta \in \Lambda.
\end{equation}
Since $B_1$ is bounded, we have that for every $\gamma > 0$ there exists $C_\gamma > 0$ such that
\begin{equation}
\label{eqe3}
 \sup_{\psi_1 \in B_1}|\widehat{\psi_1}(\xi)| \leq C_\gamma e^{-\gamma\omega(\xi)}, \qquad \forall \xi \in \mathbb{R}^d. 
 \end{equation}
Hence
\[\sup_{\psi_1 \in B_1}  I_2(\psi_1) \leq A \left \|\left( \sum_{\beta \notin \Gamma \cap \Lambda}e^{\lambda_0\omega(\beta)-\gamma_0\omega(\mu-\beta)}\right)_{\mu\in\Gamma_1\cap\Lambda}\right \|_{l^q(\Gamma_1\cap\Lambda)}\]
where $A = DC_{\gamma_0 + \lambda}$. In view of the choice of the cone $\Gamma_1$ and the constant $m$, we have that
\begin{align*}
\sup_{\psi_1 \in B_1}  I_2(\psi_1) &\leq A \left \|\left(e^{-(\gamma_0/2)\omega(\mu/m)} \sum_{\beta  \notin \Gamma \cap \Lambda}e^{\lambda_0\omega(\beta)-(\gamma_0/2)\omega(\beta/m)}\right)_{\mu\in\Gamma_1\cap\Lambda}\right \|_{l^q(\Gamma_1\cap\Lambda)}  
\\
&\leq A\left \|\left(e^{-(\gamma_0/2)\omega(\mu/m)}\right)_{\mu\in\Gamma_1\cap\Lambda}\right\|_{l^q(\Gamma_1\cap\Lambda)} \sum_{\beta \in \Lambda}e^{(\lambda_0m  -(\gamma_0/2))\omega(\beta/m)}  < \infty,
\end{align*}
provided that $\gamma_0$ is large enough.
\end{proof}

We can now proceed to show Theorem \ref{main}.

\begin{proof}[Proof of Theorem \ref{main}.] Let $I_{\Lambda}$ be a fundamental region for $\Lambda$ with $0\in I_{\Lambda}$.

$(i) \Rightarrow (ii)$: Let $\Gamma_1$ and $U_1$ be as in the first part of Lemma \ref{mainlemma}. Choose an open conic neighborhood $\Gamma_2$ of $\xi_0$ such that $\overline{\Gamma}_2 \subseteq \Gamma_1 \cup \{ 0 \}$ and $\psi \in \mathcal{D}_{(\omega)}(x_0+ U_1)$ with $\psi (x_0)\neq 0$.  Note that $B = \left \{ \psi_t := \psi e_{-t} \, : \, t \in I_{\Lambda}\right\}$ is a bounded subset of $\mathcal{D}_{(\omega)}(x_0+U_1)$. Fix $r > 0$ such
that $\Gamma_2 \cap \{ \xi \in \mathbb{R}^d \, : \, |\xi| \geq r \} \subseteq (\Gamma_1 \cap \Lambda) + I_{\Lambda}$.
Set $D = \sup_{t \in I_\Lambda}e^{\lambda \omega(t)}$. For $q < \infty$ the first part of Lemma \ref{mainlemma} implies that
\begin{align*}
\int_{\substack{ \xi \in \Gamma_2 \\ |\xi| \geq r }} |\widehat{\psi f}(\xi)v(\xi)|^q \mathrm{d}\xi & \leq \sum_{\mu\in \Gamma_1 \cap \Lambda} \int_{I_{\Lambda}}  |\widehat{\psi f}(t+\mu)v(t+\mu)|^q \mathrm{d}t \\
&\leq C^qD^q\sum_{\mu \in \Gamma_1 \cap \Lambda}v(\mu)^q \int_{I_{\Lambda}}  |\widehat{\psi_t f}(\mu)|^q \mathrm{d}t \\
&\leq|\Lambda| C^qD^q\sup_{t \in I_{\Lambda}}\sum_{\mu\in \Gamma_1 \cap \Lambda}|\widehat{\psi_t f}(\mu)v(\mu)|^q < \infty, 
\end{align*}
while for $q=\infty$ we have
\begin{align*}
\underset{|\xi| \geq r}{\sup_{\xi \in \Gamma_2}} |\widehat{\psi f}(\xi)v(\xi)| & \leq \sup_{\mu\in \Gamma_1 \cap \Lambda} \sup_{t \in I_{\Lambda}}  |\widehat{\psi f}(t+\mu)v(t+\mu)| \\
&\leq CD\sup_{t \in I_{\Lambda}}\sup_{\mu \in \Gamma_1 \cap \Lambda}|\widehat{\psi_t f}(\mu)v(\mu)| < \infty. 
\end{align*}
$(ii) \Rightarrow (i)$: The case $q=\infty$ is trivial, so we assume $q < \infty$. Let $\Gamma_1$ and $U_1$ be as in the second part of Lemma \ref{mainlemma} and let $\psi \in \mathcal{D}_{(\omega)}(x_0+ U_1)$ with $\psi (x_0)\neq 0$.  Choose an open conic neighborhood $\Gamma_2$ of $\xi_0$ such that $\overline{\Gamma}_2 \subseteq \Gamma_1 \cup \{ 0 \}$ and $ r > 0$ so large that
$(\Gamma_2 + I_\Lambda) \cap \{ \xi \in \mathbb{R}^d \, : \, |\xi| \geq r \} \subseteq \Gamma_1$. Hence
\begin{align*}
\left(\sum_{\mu \in \Gamma_2 \cap \Lambda} |\widehat{\psi f}(\mu)v(\mu)|^q \right)^{1/q} &= \left( \frac{1}{|\Lambda|}\sum_{\mu \in \Gamma_2 \cap \Lambda} \int_{\mu + I_\Lambda}|\widehat{\psi f}(\mu)v(\mu)|^q \mathrm{d}\xi \right)^{1/q} \\
&\leq |\Lambda|^{-1/q}(J_1^{1/q} + J_2^{1/q}),
\end{align*}
where
\[ J_1 := \sum_{\mu \in \Gamma_2 \cap \Lambda} \int_{\mu + I_\Lambda}|\widehat{\psi f}(\mu)-\widehat{\psi f}(\xi) |^qv(\mu)^q \mathrm{d}\xi, \]
and
\begin{align*}
 J_2 &:= \sum_{\mu \in \Gamma_2 \cap \Lambda} \int_{\mu + I_\Lambda}|\widehat{\psi f}(\xi) |^qv(\mu)^q \mathrm{d}\xi \\
&\leq \sum_{|\mu| \leq r} \int_{\mu + I_\Lambda}|\widehat{\psi f}(\xi) |^qv(\mu)^q \mathrm{d}\xi + C^q \sup_{t \in I_\Lambda} e^{q \lambda \omega(t)} \int_{\Gamma_1}|\widehat{\psi f}(\xi)v(\xi)|^q \mathrm{d}\xi < \infty.
\end{align*}
We now estimate $J_1$. Set $D = \sup_{t \in I_\Lambda} |t|$. Since $I_\Lambda$ is convex and contains the origin, we have for $\mu \in \Gamma_2 \cap \Lambda$ and $\xi \in \mu + I_\Lambda$ that
\begin{align*}
|\widehat{\psi f}(\mu)-\widehat{\psi f}(\xi) |^q &\leq |\mu - \xi|^q \sup_{t \in [0,1]} | \nabla \widehat{\psi f}(\xi + t(\mu-\xi) )|^q \\
&\leq D^q \sup_{y \in I_\Lambda} | \nabla \widehat{\psi f}(\xi - y )|^q   \\
&\leq D^q \sum_{k = 1}^d \sup_{y \in I_\Lambda} |  \mathcal{F}(x_k e_y\psi f)(\xi)|^q. 
\end{align*}
Note that the set $\{ x_k e_y\psi: \:   y \in I_\Lambda\}$ is bounded in $\mathcal{D}_{(\omega)}(x_0 + U_1)$ for each $k = 1, \ldots, d$. Hence part $(ii)'$ of Lemma \ref{mainlemma} implies that
\begin{align*}
J_1 \leq &\sum_{|\mu| \leq r} \int_{\mu + I_\Lambda}|\widehat{\psi f}(\mu)-\widehat{\psi f}(\xi) |^qv(\mu)^q \mathrm{d}\xi \\
&+  C^qD^q \sup_{t \in I_\Lambda} e^{q\lambda \omega(t)} \sum_{k = 1}^d \sup_{y \in I_\Lambda}  \int_{\Gamma_1}|\mathcal{F}(x_k e_y\psi f)(\xi)|^q v(\xi)^q\mathrm{d}\xi < \infty,
\end{align*}
which concludes the proof of the theorem.
\end{proof}
\begin{remark}\label{rk4.3} Note that if the ultradistribution $f\in\mathcal{D}_{\{\omega\}}'(\Omega)$, then the regularity requirement on the $\varphi$ used in part $(i)$ from Theorem \ref{main} and in (\ref{eqe1}) for the definition of $\mathcal{F} L_{\nu}^{p}$-microregularity can be relaxed to: $\varphi \in \mathcal{D}_{\{\omega\}}(x_{0}+U)$ and $\varphi \in \mathcal{D}_{\{\omega\}}(\Omega)$, respectively, because the test function can always be replaced by one belonging to $\mathcal{D}_{(\omega)}(\Omega)$. In addition, if $v\in\mathscr{M}_{\{\omega\}}$, Lemma \ref{mainlemma} can be strengthened: The properties $(i)'$ and $(ii)'$  hold true for bounded subsets $B$ of $\mathcal{D}_{\{\omega\}}(x_0+U_1)$. 
\end{remark}
\smallskip

The proof of Lemma \ref{mainlemma} motives the introduction of the following discrete seminorms for $\Lambda^{\ast}$-periodic $\omega$-ultradistributions. Let $\Gamma$ be a cone and $g=\sum_{\mu\in\Lambda}c_{\mu}e_{\mu}\in \mathcal{D}'_{\mathfrak{p}_{\Lambda^{*}},(\omega)}$. In analogy to (\ref{eqsFL}), we write
\begin{equation}
\label{eqsFl}
|g|_{\mathcal{F} l^{q,\Gamma}_{v,\Lambda}} :=\left\|(c_\mu)_{\mu\in\Gamma\cap \Lambda} \right\|_{l^{q}_{v}(\Gamma\cap\Lambda)}=\|(c_\mu1_\Gamma(\mu))_{\mu\in\Lambda} \|_{l^q_{v}(\Lambda)}.
\end{equation}
Using (\ref{eqsFl}), the condition $(i)$ from Theorem \ref{main} might be restated as 
$$|(\varphi f)_{\mathfrak{p}_{\Lambda^*}}|_{\mathcal{F} l^{q,\Gamma}_{v,\Lambda}}<\infty,$$
and if this is the case we shall say that $f$ is $\mathcal{F} l^q_{v,\Lambda}$-\emph{microlocally regular} at  $(x_0,\xi_0)$. The wave front set $WF_{\mathcal{F} l^q_{v,\Lambda}}(f)$ can be defined as the complement of the set of  $(x_0,\xi_0)$ such that $f$ is $\mathcal{F} l^q_{v,\Lambda}$-microlocally regular at  $(x_0,\xi_0)$.
With this terminology, we may then rephrase Theorem \ref{main} as the following equality between wave front sets: 
\begin{equation}
\label{eqWF1}
WF_{\mathcal{F} L^q_{v}}(f)= WF_{\mathcal{F} l^q_{v,\Lambda}}(f),
\end{equation}
for any lattice $\Lambda$ in $\mathbb{R}^{d}$.

We now discuss several consequences of Theorem \ref{main}. Recall that if $X$ is a linear subspace of $\mathcal{D}_{(\omega)}'(\mathbb{R}^{d})$, its associated local space (on an open subset $\Omega$ of $\mathbb{R}^d$) is $X_{loc}(\Omega)=\{f\in\mathcal{D}_{(\omega)}'(\Omega):\: \varphi f\in X, \: \forall\varphi\in\mathcal{D}_{(\omega)}(\Omega) \}$. Using the $\Lambda$-periodization operator (\ref{eqperiodization}), one may talk about local spaces with respect to vector spaces of $\Lambda$-periodic ultradistributions. Indeed, if $Y$ is a linear subspace of  $\mathcal{D}'_{\mathfrak{p}_{\Lambda},(\omega)}$, we set $Y_{loc}(\Omega)=\{f\in\mathcal{D}_{(\omega)}'(\Omega):\: (\varphi f)_{\mathfrak{p}_{\Lambda}}\in Y, \: \forall \varphi\in\mathcal{D}_{(\omega)}(\Omega) \}$. If $X$ and $Y$ are t.v.s. of ultradistributions, the topologies of $X_{loc}(\Omega)$ and $Y_{loc}(\Omega)$ can be defined in the canonical way. Employing a standard partition of the unity argument, Theorem \ref{main} immediately yields:

\begin{corollary}
\label{clocspaces} Let $v\in\mathscr{M}_{(\omega)}$. Then, $(\mathcal{F} L^q_{v})_{loc}(\Omega)=(\mathcal{F} l^q_{v,\Lambda})_{loc}(\Omega)$ topologically, for any  lattice $\Lambda$ in $\mathbb{R}^d$. 
\end{corollary}

It should be pointed out that Corollary \ref{clocspaces} remains valid even if we remove the assumption $(\alpha_0)$ on $\omega$.

The equality (\ref{eqWF1}) can be generalized to wave front sets of sup- and inf-types \cite{j-p-t-t,Pilipovictwo}. Indeed, let $(v_j)=(v_j)_{j\in J}$ and $(q_j)=(q_j)_{j\in J}$ be two indexed families with $v_j\in\mathscr{M}_{(\omega)}$ and $q_j\in[1,\infty]$, $\forall j\in J$. The wave front sets of $f\in\mathcal{D}_{(\omega)}'(\Omega)$ of inf-type with respect to the families of Fourier-Lebesgue spaces $(\mathcal{F} L^{q_j}_{v_j})_{j\in J}$ and $(\mathcal{F} l^{q_j}_{v_j,\Lambda})_{j\in J}$ are given by
$$
WF^{\text{inf}}_{(\mathcal{F} L^{q_j}_{v_j})}(f)=\bigcap_{j\in J}WF_{\mathcal{F} L^{q_j}_{v_j}}(f) \quad \mbox{and} \quad WF^{\text{inf}}_{(\mathcal{F} l^{q_j}_{v_j,\Lambda})}(f)=\bigcap_{j\in J}WF_{\mathcal{F} l^{q_j}_{v_j,\Lambda}}(f).
$$
On the other hand, the wave front set $WF^{\text{sup}}_{(\mathcal{F} L^{q_j}_{v_j})}(f)$ is defined as the complement in $\Omega\times(\mathbb{R}^d\setminus\{0\})$ of the set of points $(x_0,\xi_0)$ such that there are an open conic neighborhood $\Gamma$ of $\xi_0$ and $\varphi\in\mathcal{D}_{(\omega)}(\Omega)$ with $\varphi(x_0)\neq 0$ (both independent of the index $j$) such that $|\varphi f|_{\mathcal{F} L^{q_j,\Gamma}_{v_j}}<\infty$, for each $j\in J$.
A similar definition (employing the seminorms (\ref{eqsFl})) applies to the wave front of discrete type $WF^{\text{sup}}_{(\mathcal{F} l^{q_j}_{v_j,\Lambda})}(f)$; in this case one should also assume that $\varphi\in \mathcal{D}_{(\omega)}(x_0+U)$ where $U$ is as in Theorem \ref{main}. We obtain that following stronger version of (\ref{eqWF1}) as a corollary of the proofs of Theorem \ref{main} and Lemma \ref{mainlemma}.

\begin{theorem}\label{mainsup}
Let $v_j \in \mathscr{M}_{(\omega)}$ and $q_j \in [1, \infty]$, where $j$ runs over an index set $J$. For $f\in\mathcal{D}_{(\omega)}'(\Omega)$, we have
\begin{equation}
\label{eqWF2}
WF^{\textnormal{inf}}_{(\mathcal{F} L^{q_{j}}_{v_j})}(f)= WF^{\textnormal{inf}}_{(\mathcal{F} l^{q_j}_{v_j,\Lambda})}(f) \quad \mbox{and}\quad WF^{\textnormal{sup}}_{(\mathcal{F} L^{q_{j}}_{v_j})}(f)= WF^{\textnormal{sup}}_{(\mathcal{F} l^{q_j}_{v_j,\Lambda})}(f),
\end{equation}
for any lattice $\Lambda$ in $\mathbb{R}^{d}$.
\end{theorem}

We end this section we a discrete characterization of    $WF_{\{\omega\}}$ and $WF_{(\omega)}$. If $f\in \mathcal{D}_{(\omega)}'(\Omega)$, these wave front sets can be defined as 
$$
WF_{\{\omega\}}(f)=WF^{\textnormal{inf}}_{(\mathcal{F} L^{\infty}_{v_\lambda})}(f)  \quad \mbox{and} \quad WF_{(\omega)}(f)=WF^{\textnormal{sup}}_{(\mathcal{F} L^{\infty}_{v_\lambda})}(f),
$$
with the weights $v_{\lambda}(\xi)=e^{\lambda\omega(\xi)}$, $\lambda>0$. If  $ (x_0, \xi_0)\notin WF_{\{\omega\}}(f)$ (resp. $ (x_0, \xi_0)\notin WF_{(\omega)}(f)$), one says that $f$ is $\{\omega\}$-microlocally regular at  $(x_0, \xi_0)$ (resp. $(\omega)$-microlocally regular). We point out that when $f\in\mathcal{D}'_{\{\omega\}}(\Omega)$, the ultradistribution is $\{\omega\}$-microlocally regular ($(\omega)$-microlocally regular) at $(x_0,\xi_0)$ if and only if there are an conic open neighborhood $\Gamma$ of $\xi_0$ and $\varphi\in \mathcal{D}_{\{\omega\}}(\Omega)$ with $\varphi(x_0)\neq 0$ such that for some $\lambda > 0$ (for every $\lambda > 0$)
\begin{equation}
\label{eq4.5}  \sup_{\xi \in \Gamma} |\widehat{\varphi f}(\xi)| e^{\lambda \omega(\xi)} < \infty,  
\end{equation}
so that our definition agrees with the one used in \cite{f-g-j,Pilipovictwo,Rodino}. Summing up, we obtain the ensuing result, a corollary of (\ref{eqWF2}).

\begin{corollary}
\label{corollarywfomega}
Let $f \in \mathcal{D}_{(\omega)}'(\Omega)$ and let $\Lambda$ be a lattice in $\mathbb{R}^{d}$. Suppose that $U$ is an  open convex neighborhood of the origin such that $U\cap\Lambda^{*}=\{0\}$ and $x_0+U\subseteq\Omega$. Then, $f$ is $\{\omega\}$-microlocally regular $($$(\omega)$-microlocally regular$)$ at the point $(x_0,\xi_0)\in\Omega\times ( \mathbb{R}^d \setminus \{ 0 \})$ if and only if
there are an open conic neighborhood $\Gamma$ of $\xi_0$ and $\varphi \in \mathcal{D}_{(\omega)}(x_0+U)$ with $\varphi(x_0)\neq0$ such that for some $\lambda > 0$ $($for every $\lambda >0$$)$
\begin{equation}
\label{eq4.6} 
\sup_{\mu \in \Gamma\cap \Lambda} |\widehat{\varphi f}(\mu)| e^{\lambda \omega(\mu)} < \infty. 
\end{equation}
Moreover, if $f\in\mathcal{D}'_{\{\omega\}}(\Omega)$ the regularity assumption on the $\varphi$  witnessing $(\ref{eq4.6})$  may be relaxed to $\varphi \in \mathcal{D}_{\{\omega\}}(x_0+U)$.
\end{corollary}

Note that if a weight sequence $(M_p)_{p}$ satisfies $(M.1)$, $(M.2)$, $(M.3)'$, and $(M.4)$, then Theorem \ref{th WF non-quasianalytic} from the introduction turns out to be a particular case of Corollary \ref{corollarywfomega}; we shall however derive the general version of Theorem \ref{th WF non-quasianalytic} in Section \ref{section quasianalytic}, where the quasianalytic case will also be treated. Nonetheless, it is should be mentioned that Corollary \ref{corollarywfomega} already covers important cases such as  $\omega(\xi) = |\xi|^{1/s}$, $s > 1$, which corresponds to $M_p=(p!)^{s}$; thus it provides a discrete characterization of $\{s\}$- and $(s)$-microregularity \cite{Rodino}. Furthermore, Corollary \ref{corollarywfomega} also includes results by Ruzhansky and Turunen on the toroidal $C^{\infty}$-wave front set \cite[Thm. 7.4]{RT}, while Theorem \ref{mainsup} covers all results from \cite{d-m-s,Maksi}. 

\section{Quasianalytic wave front sets}
\label{section quasianalytic}
This last section is devoted to wave front sets defined via weight sequences. 

We consider two sequences $(M_p)_{p \in \mathbb{N}}$ and $(N_p)_{p \in \mathbb{N}}$ of positive real numbers (with $M_0 = N_0 = 1$). The sequence $(M_p)_p$ is assumed to satisfy $(M.1)$, $(M.2)'$, $(M.3)'$, while $(N_p)_p$ satisfies $(M.1)$, $(M.2)'$, and in addition
\begin{itemize}
\item [$(M.5)$]
$p!\subset N_{p}$ in the Roumieu case, \quad or \quad $ p!\prec N_p$ in the Beurling case.
\end{itemize}
The sequence $(N_p)_p$ may be quasianalytic, namely, $\sum_{p=1}^{\infty} N_{p-1}/{N_{p}}=\infty$, and will be used to measure the micro\-regularity of an $M_p$-ultradistribution. The functions $M$ and $N$ stand for the associated functions of $(M_p)_p$ and $(N_p)_p$, respectively. 

We shall follow H\"{o}rmander's approach to quasianalytic wave front sets \cite[Sect. 8.4]{Hormander}, but we slightly modify it to include wave fronts of $M_p$-ultradistributions. Let $f \in {\mathcal{D}^{(M_p)}}'(\Omega)$. We begin with the Roumieu case. The ultradistribution $f$ is said to be $\{N_p\}$-\emph{microlocally regular} at the point $(x_0,\xi_0)\in \Omega\times (\mathbb{R}^d \setminus \{ 0 \})$ if there are an open conic neighborhood $\Gamma$ of $\xi_0$, an open neighborhood $V$ of $x_0$ and a bounded sequence $(f_p)_{p \in \mathbb{N}}$  in ${\mathcal{E}^{(M_p)}}'(\Omega)$, with $f_p= f$ on $V$ for all $p \in \mathbb{N}$, such that for some $A,C > 0$ 
\begin{equation}
\label{eqWFseq}
  \sup_{\xi \in \Gamma} |\widehat{f_p} (\xi)| |\xi|^p  < AC^{p}N_p, \qquad \forall p \in \mathbb{N}.  
 \end{equation}
The wave front set $WF_{\{N_p\}}(f)$ then consists of all those $(x_0,\xi_0)\in \Omega\times (\mathbb{R}^d \setminus \{ 0 \})$ such that $f$ is \emph{not} $\{N_p\}$-microlocally regular at $(x_0,\xi_0)$. For classical Schwartz distributions $f \in {\mathcal{D}}'(\Omega)$, $WF_{\{N_p\}}(f)$ agrees with H\"{o}rmander's notion $WF_L(f)$, where $L_p = N_p^{1/p}$ (and our assumptions on $N_p$ turn out to be the same as those considered in \cite[Sect. 8.4]{Hormander} for $L_p$). In particular, for $N_p= p!$, we obtain the analytic wave front set $WF_A(f):=WF_{\{p!\}}(f)$.

 One defines $(N_p)$-\emph{microregularity} in a similar fashion (namely, by asking that (\ref{eqWFseq}) holds for all $C>0$ and some $A=A_C>0$). The definition of the Beurling wave front set  $WF_{(N_p)}(f)$ should be clear.

We need the notion of an analytic cut-off sequence \cite{Hormander,Petschze2} in order to move further. A sequence 
$(\chi_p)_{p\in \mathbb{N}}$ in $\mathcal{D}^{(M_p)}(\Omega)$ is called an $(M_p)$-analytic cut-off sequence supported in $\Omega$ if
\begin{itemize}
\item[$(a)$] $(\chi_p)_p$ is a bounded sequence in $\mathcal{D}^{(M_p)}(\Omega)$,
\item[$(b)$] $(\exists C > 0)(\forall h > 0)(\exists A_h > 0)$
\[ \|\chi_p^{(\alpha)}\|_{\mathcal{E}^{\{M_p\},h}(L)} \leq A_h(Cp)^{|\alpha|}, \qquad \forall p \in \mathbb{N},\ |\alpha| \leq p,\]
\end{itemize}
where $L \Subset \Omega$ is such that $\operatorname{supp} \chi_p \subseteq L$ for all $p \in \mathbb{N}$. We call $(\chi_p)_p$ an analytic cut-off sequence for $K\Subset\Omega$ if
\begin{itemize}
\item[$(c)$] there exists an open neighborhood $V$ of $K$ such that $\chi_p \equiv 1$ on $V$ for all $p \in \mathbb{N}$.
\end{itemize}
If $K = \{x_0\}$ is a singleton set, we shall simply say that $(\chi_p)_p$  is an analytic cut-off sequence for $x_0$. Likewise, one may also define an $\{M_p\}$-analytic cut-off as a bounded sequence $(\chi_p)_{p\in \mathbb{N}}$ in $\mathcal{D}^{\{M_p\}}(\Omega)$ such that the property $(b)$ is asked to hold just for some $h,A_h>0$.

\begin{lemma}
Let $K \Subset \mathbb{R}^d$. For every open neighborhood $W$ of $K$ there exists an $(M_p)$-analytic cut-off sequence for $K$ supported in $W$. 
\end{lemma}
\begin{proof}
 Set $d(K, \partial W) = 4\varepsilon > 0$. By \cite[Thm. 1.4.2]{Hormander} there is a sequence $(\chi_p)_p$ in $\mathcal{D}(\mathbb{R}^d)$ and $C>0$ such that
 \[ \|\chi_p^{(\alpha)}\|_{L^\infty} \leq (Cp)^{|\alpha|}, \qquad \forall p \in \mathbb{N}, |\alpha| \leq p,\]
 $\chi_p \equiv 1$ on $K+ \bar{B}(0,2\varepsilon)$ and $\operatorname{supp} \chi_p \subseteq  K+ \bar{B}(0,3\varepsilon)$ for all $p \in \mathbb{N}$. Let $\varphi \in \mathcal{D}^{(M_p)}(\mathbb{R}^d)$ with $\operatorname{supp} \varphi \subseteq B(0, \varepsilon)$ and $\int_{\mathbb{R}^d}\varphi(x) \mathrm{d}x = 1$. It is clear that $(\chi_p \ast \varphi)_p$  satisfies all requirements.
\end{proof}

 We can now state the main result of this section.
\begin{theorem}\label{mainqa}
 Let $\Lambda$ be a lattice in $\mathbb{R}^{d}$, $f \in {\mathcal{D}^{(M_p)}}'(\Omega)$, and $(x_0, \xi_0) \in \Omega \times (\mathbb{R}^d \backslash \{ 0 \})$. Suppose that $U$ is an  open convex neighborhood of the origin such that $U\cap\Lambda^{*}=\{0\}$ and $x_0+U\subset \Omega$. Then, the following statements are equivalent:\begin{itemize}
\item[$(i)$]  There are an open conic neighborhood $\Gamma$ of $\xi_0$, an open neighborhood $V \subseteq x_0+U$ of $x_0$, and a bounded sequence $(f_p)_{p \in \mathbb{N}}$  in ${\mathcal{E}^{(M_p)}}'( x_0 + U)$, with $f_p= f$ on $V$ for all $p \in \mathbb{N}$, such that for some $A,C > 0$ $($for every $C>0$ there is $A>0$$)$ 
 
\[  \sup_{\mu \in \Gamma \cap \Lambda} |\widehat{f_p} (\mu)| |\mu|^p  \leq AC^{p}N_p, \qquad \forall p \in \mathbb{N}.  \]
\item[$(ii)$] There are an open conic neighborhood $\Gamma$ of $\xi_0$ and an $(M_p)$-analytic cut-off sequence $(\chi_p)_p$ for $x_0$  supported in $x_0 + U$ such that for some $A,C > 0$ $($for every $C>0$ there is $A>0$$)$ 
\[  \sup_{\mu \in \Gamma \cap \Lambda} |\widehat{\chi_p f} (\mu)| |\mu|^p  \leq AC^{p}N_p, \qquad \forall p \in \mathbb{N}.  \]
\item[$(iii)$] There are an open conic neighborhood $\Gamma$ of $\xi_0$ and an $(M_p)$-analytic cut-off sequence $(\chi_p)_p$  for $x_0$ such that for some $A,C > 0$ $($for every $C>0$ there is $A>0$$)$  
\[  \sup_{\xi \in \Gamma} |\widehat{\chi_p f} ( \xi)| |\xi|^p  \leq AC^{p}N_p, \qquad \forall p \in \mathbb{N}.  \]
\item[$(iv)$] $f$ is $\{N_p\}$-microlocally regular $($$(N_p)$-microlocally regular$)$ at $(x_0,\xi_0)$.
\end{itemize}
\smallskip
\smallskip

\noindent If additionally $(N_p)_p$ satisfies $(M.3)'$,
then these statements are also equivalent to any of the following two conditions:
\begin{itemize}
\item [$(v)$] There are an open conic neighborhood $\Gamma$ of $\xi_0$ and $\varphi \in \mathcal{D}^{(M_p)}(x_0+U)$ with $\varphi\equiv1$ in a neighborhood of $x_0$  such that for some $r > 0$ $($for every $r >0$$)$
$$
\sup_{\mu \in \Gamma\cap \Lambda} |\widehat{\varphi f}(\mu)| e^{N(r\mu)} < \infty. 
$$
\item [$(vi)$] There are an open conic neighborhood $\Gamma$ of $\xi_0$ and $\varphi \in \mathcal{D}^{(M_p)}(x_0+U)$ with $\varphi\equiv1$ in a neighborhood of $x_0$  such that for some $r > 0$ $($for every $r >0$$)$
$$
\sup_{\xi \in \Gamma} |\widehat{\varphi f}(\xi)| e^{N(r\xi)} < \infty. 
$$
\end{itemize}
\end{theorem}

Observe that both Theorem \ref{th WF non-quasianalytic} and Theorem \ref{th WF analytic} are contained in Theorem \ref{mainqa}. The rest of this section is dedicated to give a proof of Theorem \ref{mainqa}. We divide its proof into several intermediate lemmas.
\begin{lemma}\label{bounded}
Let $(\chi_p)_p$ be an $(M_p)$-analytic cut-off sequence and $B$ a bounded subset of $\mathcal{E}^{(p!)}(\mathbb{R}^d)$. Then,
\begin{itemize}
\item[\mbox{}] $(\exists C > 0)(\forall h > 0)(\exists C_h > 0)$
\[ \sup_{\psi \in B}\sup_{\xi \in \mathbb{R}^d} |\widehat{ \chi_p\psi}(\xi)| |\xi|^{p} e^{M(|\xi|/h)} \leq C_h (Cp)^p, \qquad \forall  p \in \mathbb{N}.\]
\end{itemize}
\end{lemma}
\begin{proof}
Let $K \Subset \mathbb{R}^d$ be such that $\operatorname{supp} \chi_p \subseteq K$ for all $p \in \mathbb{N}$. A straightforward computation yields
\begin{itemize}
\item[\mbox{}] $(\exists C > 0)(\forall h > 0)(\exists C_h > 0)$
\[ \sup_{\psi \in B} \|(\chi_p\psi)^{(\alpha)}\|_{\mathcal{E}^{(M_p),h}(K)} \leq C_h (Cp)^{|\alpha|}, \qquad \forall p \in \mathbb{N}, |\alpha| \leq p.\]
\end{itemize}
Hence for all $p, q \in \mathbb{N}$, $\psi \in B$ and $h > 0$ we have that
\begin{align*}
|\widehat{ \chi_p\psi}(\xi)| |\xi|^{p+q} &\leq d^{\frac{p+q}{2}} |K| \sup_{|\alpha| = p} \sup_{|\beta|=q} \sup_{x \in K}|(\chi_p\psi)^{(\alpha+\beta)}(x)| \\
 &\leq |K| C_{h/\sqrt{d}}(\sqrt{d}Cp)^p h^qM_q,
 \end{align*}
 where $|K|$ is the Lebesgue measure of $K$. Consequently,
 \[ |\widehat{ \chi_p\psi}(\xi)| |\xi|^{p} \leq |K| C_{h/\sqrt{d}}(\sqrt{d}Cp)^p \inf_{q \in \mathbb{N}} \frac{h^qM_q}{|\xi|^q} = |K|  C_{h/\sqrt{d}}(\sqrt{d}Cp)^pe^{-M(|\xi|/h)}.\]
\end{proof}
\begin{lemma}\label{lemmaqa} 
Condition $(ii)$ from Theorem \ref{mainqa} implies that there are an open conic neighborhood $\Gamma_1$ of $\xi_0$, an open neighborhood $U_1 \subseteq U$ of the origin, and an $(M_p)$-analytic cut-off sequence $(\kappa_p)_p$ for $x_0$ supported in $x_0 + U_1$ such that for every bounded set $B$ in $\mathcal{E}^{(p!)}(\mathbb{R}^d)$ there are $A,C > 0$ $($for every $C>0$ there is $A>0$$)$  
for which
\[  \sup_{\psi \in B}\sup_{\mu \in \Gamma_1 \cap \Lambda}  |\widehat{\kappa_p \psi f} (\mu)| |\mu|^p  \leq A C^{p}N_p, \qquad \forall p \in \mathbb{N}.  \]
\end{lemma}
\begin{proof}
We choose an open conic neighborhood $\Gamma_1$ of $\xi_0$ such that $\overline{\Gamma}_1 \subseteq \Gamma \cup \{ 0 \}$. Let  $0 < c < 1$ be smaller than the distance between $\partial \Gamma$ and the intersection of $\Gamma_1$ with the unit sphere. Hence $\{y \in \mathbb{R}^d \,:\, (\exists \xi \in \Gamma_1)(|\xi -y|\leq c|\xi|) \} \subseteq \Gamma$. Next, choose an open neighborhood $U_1 \subseteq U$ of the origin such that $\chi_p \equiv  1$ on  $x_0 + U_1$ for all $p \in \mathbb{N}$. Let  $(\kappa_p)_p$  be an analytic cut-off sequence for $x_0$ supported in $x_0 + U_1$. Let $B$ be a bounded set in $\mathcal{E}^{(p!)}(\mathbb{R}^d)$. Since $x_0+U$ is contained in some fundamental region $I_{\Lambda^{\ast}}$ of the dual lattice $\Lambda^*$,we have that for an arbitrary $g \in \mathcal{E}^{(M_p)'}(\mathbb{R}^{d})$ having support in  the interior of $I_{\Lambda^{\ast}}$ the Fourier coefficients of its $\Lambda^*$-periodization are given by $(\widehat{g}(\mu))_{\mu\in\Lambda}$. In particular, using the fact that $(\kappa_
 p \psi f
 )_{\mathfrak{p}_{\Lambda^*}}= (\kappa_p\psi)_{\mathfrak{p}_{\Lambda^*}}(\chi_p f)_{\mathfrak{p}_{\Lambda^*}}$ for all $p \in \mathbb{N}$ and $\psi \in B$, we obtain that
\[
 |\widehat{\kappa_p \psi f} (\mu)| = \left| \sum_{\beta \in \Lambda} \widehat{\chi_pf}(\mu - \beta) \widehat{\kappa_p\psi}(\beta)\right| \leq I_{1,p}(\mu) + I_{2,p}(\mu),
\]
where
\[
I_{1,p}(\mu) = \underset{\beta \in \Lambda}{\sum_{|\beta| \leq c|\mu|}} |\widehat{\chi_pf}(\mu-\beta)| |\widehat{\kappa_p\psi}(\beta)|
\]
and 
\[
I_{2,p}(\mu) = \underset{\beta \in \Lambda}{\sum_{|\beta| > c|\mu|}} |\widehat{\chi_pf}(\mu-\beta)| |\widehat{\kappa_p\psi}(\beta)|
.
\]
Due to the fact that the set $\{\kappa_p\psi \,: \, p \in \mathbb{N}, \psi \in B\}$ is bounded in $\mathcal{D}^{(M_p)}(\Omega)$, there exists $D > 0$, independent of $p$ and $\psi$, such that
\[ I_{1,p}(\mu) = \underset{\beta \in \Lambda}{\sum_{|\mu-\beta| \leq c|\mu|}} |\widehat{\chi_pf}(\beta)| |\widehat{\kappa_p\psi}(\mu-\beta)| \leq D \underset{\beta \in \Lambda}{\sup_{|\mu-\beta| \leq c|\mu|}} |\widehat{\chi_pf}(\beta)| 
 \]
 Since $|\mu-\beta| \leq c|\mu|$ implies $|\beta| \geq (1-c)|\mu|$, we have that
 \begin{align*}
  \sup_{\psi \in B}\sup_{\mu \in \Gamma_1 \cap \Lambda} I_{1,p}(\mu) |\mu|^p &\leq D \sup_{\mu \in \Gamma_1 \cap \Lambda} |\mu|^p \sup_{\beta \in \Lambda}\sup_{|\mu-\beta| \leq c|\mu|} |\widehat{\chi_pf}(\beta)|\\
  & \leq \frac{D}{(1-c)^p} \sup_{\beta \in \Gamma \cap \Lambda}|\beta|^p|\widehat{\chi_pf}(\beta)| \\
  & \leq D A \left(\frac{C}{1-c} \right)^p N_p.
 \end{align*}
We now estimate $I_{2,p}(\mu)$.  The boundedness of $(\chi_p)_p$ implies that there are $D', h > 0$ such that for all $p  \in \mathbb{N}$
 \[  |\widehat{\chi_pf}(\xi)| \leq D' e^{M(|\xi|/h)}, \qquad \forall \xi \in \mathbb{R}^d.\]
Since $|\mu-\beta| \leq (1+ c^{-1})|\beta|$ for $|\beta| \geq c |\mu|$, we have that
\begin{align*}
\sup_{\psi \in B}\sup_{\mu \in \Gamma_1 \cap \Lambda} I_{2,p}(\mu) |\mu|^p &\leq D'  \sup_{\psi \in B}\sup_{\mu \in \Gamma_1 \cap \Lambda} |\mu|^p \underset{\beta \in \Lambda}{\sum_{|\beta| \geq c|\mu|} }e^{M(|\mu-\beta|/h)}|\widehat{\kappa_p\psi}(\beta)|  
\\
&
\leq   \frac{D'}{c^p} \sup_{\psi \in B}\sum_{\beta \in \Lambda} e^{M((1+ c^{-1})|\beta|/h)} |\beta|^p|\widehat{\kappa_p\psi}(\beta)|. 
 \end{align*}
 The result now follows from Lemma \ref{bounded}, \cite[Prop. 3.4]{Komatsu}, and the assumption $(M.5)$.
\end{proof}
\begin{proof}[Proof of Theorem \ref{mainqa}.] The implications $(iii) \Rightarrow (iv)$ and 
$(iv) \Rightarrow (i)$, $(v)\Rightarrow(i)$, and $(vi) \Rightarrow(v)$ are trivial. The proof of
$(i) \Rightarrow (ii)$ is similar to the first part of that of \cite[Lemma 8.4.4]{Hormander}, so we omit it. Let us now show
$(ii) \Rightarrow (iii)$. For it,  let $I_\Lambda$ be a fundamental region for $\Lambda$ with $0 \in I_\Lambda
$ and let $\Gamma_1$ and $(\kappa_p)_p$ be as in Lemma \ref{lemmaqa}. Choose an open conic neighborhood $\Gamma_2$ of $\xi_0$ such that $\overline{\Gamma}_2 \subseteq \Gamma_1 \cup \{ 0 \}$.  Fix $r > 0$ such
that $\Gamma_2 \cap \{ \xi \in \mathbb{R}^d \, | \, |\xi| \geq r \} \subseteq (\Gamma_1 \cap \Lambda) + I_\Lambda$. Since $I_{\Lambda}$ is bounded, there is $D > 0$ such that $|\mu+t|^p \leq D^p |\mu|^p$ for all $p \in \mathbb{N}$, $t \in I_\Lambda$ and $\mu \in \Lambda\setminus\{0\}$.
Hence Lemma \ref{lemmaqa} and the fact that $B = \left \{ e_{-t} \, | \, t \in I_\Lambda \right\}$ is a bounded subset of $\mathcal{E}^{(p!)}(\mathbb{R}^d)$ imply that
\begin{align*}
\sup_{\xi \in \Gamma_2, |\xi| \geq r } |\widehat{\kappa_p f}(\xi)||\xi|^p & \leq \sup_{\mu \in \Gamma_1 \cap \Lambda} \sup_{t \in I_\Lambda}  |\widehat{\kappa_p f}(\mu+t)||\mu+t|^p \\
&\leq D^p\sup_{t \in I_\Lambda}\sup_{\mu \in \Gamma_1 \cap \Lambda}|\mathcal{F}(\kappa_p e_{-t} f)(\mu)||\mu|^p \leq A(DC)^pN_p.
\end{align*}

It remains to establish $(iv) \Rightarrow(vi)$ under the additional assumption that $(N_p)_p$ satisfies $(M.3)'$. So assume $(iv)$ and let $\Gamma_1$ and $c$ be as in the proof of Lemma  \ref{lemmaqa}. We use the auxiliary weight sequence $Q_p=\min_{0\leq q\leq p}M_q N_{p-q}$, with associated function $Q$. A result by Roumieu (\cite[Lemm. 3.5]{Komatsu}) ensures that $(Q_p)_p$ satisfies $(M.1)$ and that $Q=M+N$; thus \cite[Prop. 3.4]{Komatsu} implies that $(Q_p)_p$ also fulfills $(M.2)'$. Select any $V_1$ with $\overline{V}_1\Subset V$ and $\varphi\in \mathcal{D}^{(Q_p)}(V)$ with $\varphi\equiv1$ on $V_1$. Furthermore, using the boundedness of $(f_p)_{p}$, find constants $D, h>0$ such that 
 \[ \|\widehat{\varphi}\|_{L^{1}}\leq D \quad \mbox{and} 
 \quad \sup_{\xi\in\mathbb{R}^{d},\: p\in\mathbb{N} } e^{-M(|\xi|/h)}|\widehat{f_p}(\xi)| \leq D.
 \]
Set $R=\max{\{ h^{-1}(1+ c^{-1}),(cC)^{-1}(1-c)\}}.$ Reasoning as in the proof of Lemma \ref{lemmaqa}, we obtain, for all $p\in\mathbb{N}$ and $\xi\in \Gamma_1$,
\begin{align*}
\left| \xi\right|^{p}|\widehat{\varphi f}(\xi)|&
=\left| \xi\right|^{p}|\widehat{\varphi f_p}(\xi)|
\\
&\leq
DA \left(\frac{C}{1-c}\right)^{p} N_p+DN_p\int_{|\eta| > c |\xi| }e^{M(h^{-1}|\xi - \eta|)}\frac{|\eta|^p}{c^{p}N_p} |\widehat{\varphi}(\eta)|\mathrm{d}\eta
\\
&\leq
D\left(\frac{C}{1-c}\right)^{p} N_p\left(A +\int_{\mathbb{R}^d}e^{M(h^{-1}(1+ c^{-1})|\eta|)+N((cC)^{-1}(1-c)|\eta|)} |\widehat{\varphi}(\eta)| \mathrm{d}\eta\right)
\\
&\leq
D\left(\frac{C}{1-c}\right)^{p} N_p\left(A +\int_{\mathbb{R}^d}e^{Q(R|\eta|)} |\widehat{\varphi}(\eta)|\mathrm{d}\eta \right).
\end{align*}
The last integral is finite because of \cite[Prop. 3.4]{Komatsu} and so $\displaystyle\sup_{\xi \in \Gamma_1} |\widehat{\varphi f}(\xi)| e^{N(r\xi)} < \infty$, with $r=(1-c)/C$.
\end{proof}

We conclude the article with some remarks.
\begin{remark}\label{rk1Fs} Let $f\in {\mathcal{D}^{\{M_p\}}}'(\Omega)$ ($f\in \mathcal{D}'(\Omega)$). Then, the requirements on the sequence $(\chi_p)_p$ and the test function $\varphi$ from Theorem \ref{mainqa} can be relaxed to:   $(\chi_p)_p$ is an $\{M_p\}$-analytic (analytic) cut-off sequence  for $x_0$ and $\varphi \in \mathcal{D}^{\{M_p\}}(x_0+U)$ ($\varphi \in \mathcal{D}(x_0+U)$). Furthermore, the properties of $(f_p)_p$ in condition $(i)$ of Theorem \ref{mainqa} and in the definition of $\{N_p\}$- and $(N_p)$-microregularity  can be strengthened to: $(f_p)_p$ is bounded in ${\mathcal{E}^{\{M_p\}}}'(x_0+U)$ (in $\mathcal{E}'(x_0+U)$) and ${\mathcal{E}^{\{M_p\}}}'(\Omega)$ ($\mathcal{E}'(\Omega)$), respectively. The proofs of these assertions are straightforward modifications of the arguments given in this section and are therefore left to the reader.
\end{remark}

\begin{remark}
If the sequence  $(M_p)_p$ additionally satisfies $(M.4)''$, then one might simply ask $\varphi(x_0)\neq0$ for $\varphi \in \mathcal{D}^{(M_p)}(x_0+U)$ ($\varphi \in \mathcal{D}^{\{M_p\}}(x_0+U)$ if $f\in {\mathcal{D}^{\{M_p\}}}'(\Omega)$)  occurring in $(v)$ from Theorem \ref{mainqa}. In fact, all this is a just consequence of the fact that $\mathcal{E}^{\{M_p\}}(\Omega)$ and $\mathcal{E}^{(M_p)}(\Omega)$ become inverse closed under $(M.4)''$.  In particular, that is the case when $(M_p)_p$ fulfills $(M.3)$, as already pointed out in Subsection \ref{subsection sequences}. 
\end{remark}


\begin{thebibliography}{12}


\bibitem{a-j-o} A.~A.~Albanese, D.~Jornet, A.~Oliaro, \emph{Wave front sets for ultradistribution solutions of linear partial differential operators with coefficients in non-quasianalytic classes,} Math. Nachr. \textbf{285} (2012), 411--425. 

\bibitem{beurling} A.~Beurling, \emph{Sur les int\'{e}grales de Fourier absolument convergentes et leur application \`{a} une transformation fonctionelle}, in: IX Congr. Math. Scand., pp. 345--366, Helsingfors, 1938.

\bibitem{Bjorck} G.~Bj\"orck, \emph{Linear partial differential operators and generalized distributions}, Ark. Mat. \textbf{6} (1966), 351--407.

\bibitem{b-m-t} R.~W.~Braun, R.~Meise, B.~A.~Taylor, \emph{Ultradifferentiable functions and Fourier analysis,} Result. Math. \textbf{17} (1990), 206--237.

\bibitem{b-d-h} C.~Brouder, N.~V.~Dang, F.~H\'{e}lein, \emph{A smooth introduction to the wavefront set,} J. Phys. A \textbf{47} (2014), 443001, 30 pp. 

\bibitem{c-j-t1}S.~Coriasco, K.~Johansson, J.~Toft, \emph{Local wave-front sets of Banach and Fr\'{e}chet types, and pseudo-differential operators,} Monatsh. Math. \textbf{169} (2013), 285--316. 

\bibitem{c-j-t2} S.~Coriasco, K.~Johansson, J.~Toft, \emph{Global wave-front sets of Banach, Fr\'{e}chet and modulation space types, and pseudo-differential operators,} J. Differential Equations \textbf{254} (2013), 3228--3258.

\bibitem{d-m-s} D.~Doli\'{c}anin-Djeki\'{c}, S.~Maksimovi\'{c}, P. Sokoloski, \emph{Wave fronts of ultradistributions via Fourier series coefficients,}
Novi Sad J. Math, to appear. 

\bibitem{f-g-j} C. Fern\'{a}ndez, A. Galbis, D. Jornet, \emph{Pseudodifferential operators of Beurling type and the wave front set,} J. Math. Anal. Appl. \textbf{340} (2008), 1153--1170. 

\bibitem{G-R-S} I. Gelfand, D. Raikov, G. Shilov, \emph{Commutative normed rings,} Chelsea Publishing Co., New York, 1964.

\bibitem{Gorba} V.~I.~Gorbachuk, \emph{Fourier series of periodic ultradistributions}, Ukrain. Mat. Zh. \textbf{34} (1982), 118--123.

\bibitem{hilleph} E.~Hille, R.~S.~Phillips, \emph{Functional analysis and semi-groups,} American Mathematical Society, Providence, R.I, 1974.

\bibitem{Hormander1971} L.~H\"{o}rmander, \emph{Fourier integral operators. I,} Acta Math. \textbf{127} (1971), 79--183. 

\bibitem{Hormander} L.~H\"ormander, \emph{The analysis of linear partial differential operators. I. Distribution
theory and Fourier analysis}, Second edition, Springer-Verlag, Berlin, 1990.

\bibitem{Johansson} K. Johansson, S. Pilipovi\'{c}, N. Teofanov, J. Toft, \emph{Gabor pairs, and a discrete approach to wave-front sets}, Monatsh. Math. \textbf{166} (2012), 181--199.

\bibitem{j-p-t-t} K.~Johansson, S.~Pilipovi\'{c}, N.~Teofanov, J.~Toft, \emph{Micro-local analysis in some spaces of ultradistributions}, Publ. Inst. Math. (Beograd) (N.S.) \textbf{92} (2012), 1--24.

\bibitem{Komatsu} H. Komatsu, \emph{Ultradistributions I. Structure theorems and a characterization}, J. Fac. Sci. Tokyo Sect. IA Math. \textbf{20} (1973), 25--105.

\bibitem{Komatsutwee} H.~Komatsu, \emph{Microlocal  analysis  in  Gevrey  classes  and  in  complex  domains,} in: Microlocal
analysis and applications, pp. 161--236, Springer, Berlin, 1991.

\bibitem{Maksi} S.~Maksimovi\'{c}, S.~Pilipovi\'{c}, P.~Sokoloski, J.~Vindas, \emph{Wave fronts via Fourier series coefficients}, Publ. Inst. Math. (Beograd) (N.S.) \textbf{97} (2015), 1--10. 

\bibitem{Petzsche78}H.-J.~Petzsche,
\emph{Die Nuklearit\"{a}t der Ultradistributionsr\"{a}ume und der Satz vom Kern. I,} 
Manuscripta Math. \textbf{24} (1978), 133--171. 

\bibitem{Petschze2} H.-J.~Petzsche, \emph{Approximation of ultradifferentiable functions by polynomials and entire functions}, Manuscripta Math. \textbf{48} (1984), 227--250. 

\bibitem{Petzsche88} H.-J.~ Petzsche, \emph{On E. Borel's theorem,} Math. Ann. \textbf{282} (1988), 299--313.

\bibitem{Petzsche-V} H.-J.~Petzsche, D.~Vogt, \emph{Almost analytic extension of ultradifferentiable functions and the boundary values of holomorphic functions,} Math. Ann. \textbf{267} (1984), 17--35.


\bibitem{Pilipovicthree} S.~Pilipovi\'{c}, N.~Teofanov, J.~Toft, 
\emph{Micro-local analysis in Fourier Lebesgue and modulation spaces: part II}, J. Pseudo-Differ. Oper. Appl. \textbf{1} (2010), 341--376.

\bibitem{Pilipovictwo} S.~Pilipovi\'{c}, N.~Teofanov, J.~Toft, \emph{Micro-local analysis with Fourier-Lebesgue spaces. Part I}, J. Fourier Anal. Appl. \textbf{17} (2011), 374--407.

\bibitem{Rainer-S} A.~Rainer, G.~Schindl, \emph{Composition in ultradifferentiable classes,} Studia Math. \textbf{224} (2014), 97--131.

\bibitem{Rodino} L.~Rodino, \emph{Linear partial differential operators in Gevrey spaces}, World Scientific Publishing Co.,  Inc., River Edge, NJ, 1993.

\bibitem{Rodino-W} L.~Rodino, P.~Wahlberg, \emph{The Gabor wave front set,} Monatsh. Math. \textbf{173} (2014), 625--655. 

\bibitem{Rudin62} W.~Rudin, \emph{Division in algebras of infinitely differentiable functions,} J. Math. Mech. \textbf{11} (1962), 797--809. 

\bibitem{RT} M. Ruzhansky, V. Turunen,
\emph{ Quantization of pseudo-differential operators on the torus}, J. Fourier Anal. Appl. \textbf{16} (2010), 943--982.

\bibitem{Ruzh} M. Ruzhansky, V. Turunen, \emph{Pseudo-differential operators and symmetries. Background analysis and advanced topics}, Birkh\" auser Verlag, Basel, 2010.

\bibitem{zeeman} A.~H.~Zemanian, \emph{Distribution theory and transform analysis. An introduction to generalized functions, with applications}, Second edition, Dover Publications, Inc., New York, 1987. 

\end{thebibliography}
\end{document}